\PassOptionsToPackage{dvipsnames}{xcolor}
\documentclass[onefignum,onetabnum,a4paper]{siamart190516}
\usepackage[utf8]{inputenc}
\usepackage[T1]{fontenc}
\usepackage{amssymb}
\usepackage{amsfonts}
\usepackage{amsmath}
\usepackage{graphicx}
\usepackage{tikz}
\usepackage{pgfplots}
\usepackage{subfig} 
\usepackage{textcomp}
\usepackage{mathtools}
\usepackage{fullpage}

\usepackage{moreverb}
\usepackage{enumerate}
\usepackage{aligned-overset}
\usepackage[most]{tcolorbox}

\newenvironment{numberedproof}[1]{{\textit{Proof #1}:}}{\hfill\proofbox\par\bigskip}

\newcommand{\R}{\mathbb{R}}
\newcommand{\Rpos}{\mathbb{R}_{+}}
\newcommand{\N}{\mathbb{N}}
\newcommand{\Z}{\mathbb{Z}}
\newcommand{\tUp}{\widetilde{U}^{(p)}}

\newcommand{\tNp}{\widetilde{N}^{(p)}}
\newcommand{\tNpmone}{\widetilde{N}^{(p-1)}}
\newcommand{\tNppar}{\widetilde{N}^{(\ppar)}}

\newcommand{\Npperpij}{\widehat{N}_{i,j}^{(\pperp)}}
\newcommand{\Npperpi}{\widehat{N}_{i}^{(\pperp)}}
\newcommand{\tFp}{\widetilde{F}^{(p)}}
\newcommand{\tfp}{\widetilde{f}^{(p)}}
\newcommand{\Dpar}{D_{x_\parallel}}
\newcommand{\Uparp}{U_{x_\parallel}^{(p)}}
\newcommand{\Uparpmone}{U_{x_\parallel}^{(p-1)}}
\newcommand{\Fparp}{F_{x_\parallel}^{(p)}}
\newcommand{\fparp}{f_{x_\parallel}^{(p)}}

\newcommand{\ppar}{p_{\parallel}}
\newcommand{\qpar}{q_{\parallel}}
\newcommand{\pperp}{p_{\perp}}
\newcommand{\qperp}{q_{\perp}}

\newcommand{\Cloctwo}{C_{\mathrm{loc},2}}
\newcommand{\Creg}{C_{\mathrm{reg}}}
\newcommand{\CregN}{C_{\mathrm{reg,N}}}
\newcommand{\omegaeps}{\xi}
\newcommand{\omegac}{\omega_{\mathbf{v}}}
\newcommand{\omegacprime}{\omega_{\mathbf{v'}}}
\newcommand{\omegace}{\omega_{\mathbf{ve}}}
\newcommand{\omegaceprime}{\omega_{\mathbf{v'e'}}}
\newcommand{\omegae}{\omega_{\mathbf{e}}}
\newcommand{\omegaeprime}{\omega_{\mathbf{e'}}}
\newcommand{\eremk}{\hbox{}\hfill\rule{0.8ex}{0.8ex}}

\newsiamremark{remark}{Remark}

\numberwithin{equation}{section}

\newcommand{\abs}[1]{\left\vert #1 \right\vert}
\newcommand{\skp}[1]{\left< #1 \right>}
\newcommand{\norm}[1]{\left\| #1 \right\|}
\newcommand{\supp}{\operatorname*{supp}}

\newcommand{\deta}{\partial_x^\eta}
\newcommand{\etam}{|\eta|}
\newcommand{\dbeta}{\partial_x^\beta}
\newcommand{\betam}{|\beta|}
\newcommand{\CacInt}{C_{\mathrm{int}}}

\title{Weighted analytic regularity for the integral fractional Laplacian in
  polygons}
\author{%
  Markus Faustmann%
  \thanks{%
    Institut für Analysis und Scientific Computing, TU Wien, A-1040 Wien, Austria%
  }%
  \and 
  Carlo Marcati%
  \thanks{Dipartimento di Matematica ``F. Casorati'', Universit{\`a} di Pavia,
    I-27100 Pavia, Italy%
      }%
  \and
  Jens Markus Melenk\footnotemark[1]%
  \and
  Christoph Schwab%
  \thanks{Seminar for Applied Mathematics, ETH Zurich, CH-8092 Zürich,%
    Switzerland%
  \funding{The research of JMM is funded by the Austrian Science Fund (FWF) by
      the special research program {\it Taming complexity in PDE systems} (grant
      SFB F65). The research of CM was performed during a PostDoctoral fellowship
      at the Seminar for Applied Mathematics, ETH Z\"urich, in 2020-2021.}%
}
}
\begin{document}

\maketitle
\centerline{\emph{dedicated to Professor Ivo M.~Babu{\v s}ka on the occasion of his $60_{16}$th birthday}}

\begin{abstract}
  We prove weighted analytic regularity of solutions to the Dirichlet problem
  for the integral fractional Laplacian in polygons with analytic right-hand side.
  We localize the problem through the Caffarelli-Silvestre extension and study
  the tangential differentiability of the extended solutions, followed 
  by bootstrapping based on Caccioppoli inequalities 
  on dyadic decompositions of vertex, edge, and vertex-edge neighborhoods.
\end{abstract}
\section{Introduction}
In this work, we study the regularity of solutions to the Dirichlet problem for
the integral fractional Laplacian
\begin{equation}
  \label{eq:intro-eq}
  (-\Delta)^s u = f \text{ on }\Omega, \qquad u = 0 \text{ on } \R^d \setminus \overline{\Omega},
\end{equation}
with $0<s<1$. 
We consider the case of a polygonal $\Omega$ and a source term $f$ that is analytic 
in $\overline{\Omega}$,
and derive weighted analytic-type estimates for the solution $u$, with
vertex and edge weights that vanish on the domain boundary $\partial\Omega$.

Unlike their integer order counterparts, solutions to fractional Laplace
equations are known to lose regularity near $\partial\Omega$, 
even when the source term and $\partial\Omega$ are smooth (see, e.g., \cite{Grubb15}). 
After the establishment of low-order Hölder regularity up to the boundary for
$C^{1,1}$ domains in \cite{ros-oton-serra14}, solutions to the Dirichlet problem
for the integral fractional Laplacian have been shown to be smooth (after
removal of the boundary singularity) in $C^\infty$ domains \cite{Grubb15}.
Subsequent results have filled in the gap between low and high regularity in
Sobolev \cite{Abels2020} and Hölder spaces \cite{Abatangelo2020}, with
appropriate assumptions on the regularity of the domain.
Besov regularity of weak solutions $u$ of \eqref{eq:intro-eq} 
has recently been established in \cite{BN21} in Lipschitz domains $\Omega$.
Finally, for polygonal $\Omega$, the precise characterization of the
singularities of the solution in vertex, edge, and
vertex-edge neighborhoods is the focus of the Mellin-based analysis of \cite{GiEPSSt,stocek}.

For smooth boundary $\partial\Omega$, 
\cite{Grubb15} characterizes the mapping properties 
of the integral fractional Laplacian, exhibiting in particular 
the anisotropic nature of solutions near the boundary. 
Interior regularity results
have been obtained in \cite{cozzi17,biccari-warma-zuazua17,faustmann-karkulik-melenk22} and, 
under analyticity assumptions
on the right-hand side, (interior) analyticity of the solution has been derived
even for certain nonlinear problems \cite{koch-rueland19,dallacqua-etal12,dallacqua-etal13} and more general 
integro-differential operators \cite{GFV2015}.
The loss of regularity 
near $\partial\Omega$ 
can be accounted for by weights in the context of isotropic Sobolev spaces \cite{acosta-borthagaray17}. 
While all the latter references focus on the Dirichlet integral fractional Laplacian, 
which is also the topic of the present work,  
corresponding regularity results for the Dirichlet spectral fractional Laplacian are also available, 
see, e.g., \cite{caffarelli-stinga16}. 

The purpose of the present work is a description of the regularity of the solution
of (\ref{eq:intro-eq}) for piecewise analytic input data that reflects 
both the interior analyticity and the anisotropic nature of the solution near the boundary. 
This is achieved in Theorem~\ref{thm:mainresult} through the use of appropriately weighted Sobolev spaces. 
Unlike local elliptic operators in polygons, 
    for which vertex-weighted spaces allow for analytic regularity shifts (e.g., \cite{BabGuoI,MR162}), corresponding results for
    fractional operators in polygons require additionally edge-weights \cite{Grubb15}.

An observation that was influential in the analysis of elliptic fractional diffusion problems
is their \emph{localization through a local, divergence form, elliptic degenerate operator
in higher dimension}. 
First pointed out in \cite{CafSil07}, it subsequently inspired many developments 
in the analysis of fractional problems. 
While not falling into the standard elliptic setting (see, e.g., the discussion in \cite{Grubb15}),
the localization via a higher-dimensional local elliptic boundary value problem does 
allow one to leverage tools from elliptic regularity theory. Indeed, the present work studies
the regularity of the higher-dimensional local degenerate elliptic problem and infers 
from that the regularity of (\ref{eq:intro-eq}) by taking appropriate traces. 

Our analysis is based on Caccioppoli estimates and bootstrapping methods for the 
higher-dimensional elliptic problem. 
Such arguments are well-known to require (under suitable assumptions on the data)
a basic regularity shift for variational solutions 
from the energy space of the problem 
(in the present case, a fractional order, nonweighted Sobolev space)
into a slightly smaller subspace (with a fixed order increase in regularity).
This is subsequently used to iterate in a bootstrapping manner local regularity estimates 
of Caccioppoli type on appropriately scaled balls in a Besicovitch covering of the domain. 
In the classical setting of non-degenerate elliptic problems, the 
initial regularity shift (into a vertex-weighted Sobolev space)
is achieved by localization and a Mellin type analysis at vertices,
as presented, e.g., in \cite{MR162} and the references there. The subsequent bootstrapping
can then lead to analytic regularity as established in a number of references for local 
non-degenerate elliptic boundary value problems  
(see, e.g., \cite{BabGuoI,BabGuo3dI,BabGuo3dII,CoDaNi} and the references there). 
The bootstrapping argument of the present work structurally follows these approaches. 

While delivering sharp ranges of indices for regularity shifts (as limited by 
poles in the Mellin resolvent), the Mellin-based approach will naturally meet with difficulties 
in settings with multiple, non-separated vertices 
(as arise, e.g., in general Lipschitz polygons). 
Here, an alternative approach to extract 
some finite amount of regularity in nonweighted Besov-Triebel-Lizorkin spaces was proposed 
in \cite{Savare}; it is based on difference-quotient techniques and compactness arguments.
In the present work, our initial regularity shift is obtained with the techniques 
of \cite{Savare}. 
In contrast to the Mellin approach, the technique of \cite{Savare} leads to regularity shifts even in Lipschitz domains 
but does not, as a rule, give better shifts for piecewise smooth geometries such as polygons. 
While this could be viewed as mathematically non-satisfactory, 
we argue in the present note that it can be quite adequate 
as a base shift estimate in establishing analytic regularity in 
vertex- and boundary-weighted Sobolev spaces, where quantitative 
control of constants under scaling takes precedence over the 
optimal range of smoothness indices.

\subsection{Impact on numerical methods}

The mathematical analysis of efficient numerical methods for the numerical approximation of
fractional diffusion has received considerable attention in recent years. 
We only mention the surveys  \cite{GunbActa,BBNOS18,BLN19,LPGSGZMCMAK18} 
and the references there for broad surveys on recent developments in the analysis and in 
the discretization of nonlocal, fractional models. At this point, most basic issues in 
the numerical analysis of discretizations of linear, elliptic fractional
diffusion problems are rather well understood, and convergence rates of 
variational discretizations based on finite element methods on 
regular simplicial meshes have been established, subject to appropriate regularity hypotheses. 
Regularity in isotropic Sobolev/Besov spaces is available, \cite{BN21}, leading to certain 
algebraically convergent methods based on shape-regular simplicial meshes. 
As discussed above, the expected solution behavior is anisotropic so that edge-refined 
meshes can lead to improved convergence rates. Indeed, a sharp analysis of vertex and edge 
singularities via Mellin techniques is the purpose of \cite{GiEPSSt,stocek} and allows for 
unravelling the optimal mesh grading for algebraically convergent methods. The analytic 
regularity result obtained in Theorem~\ref{thm:mainresult} captures both the anisotropic 
behavior of the solution and its analyticity so that \emph{exponentially convergent} 
numerical methods for integral fractional Laplace equations in polygons can be developed 
in our follow-up work \cite{FMMS-hp}; see also \cite{FMMS22} for the corresponding convergence theory in 1D.

\subsection{Structure of this text}
\label{sec:Struct}
After having introduced some basic notation in the forthcoming subsection, in Section \ref{sec:Set}
we present the variational formulation
of the nonlocal boundary value problem. 
We also introduce the scales of boundary-weighted Sobolev spaces on which our
regularity analysis is based. 
In Section~\ref{sec:WgtAnRegR2}, we state our main regularity result, 
Theorem~\ref{thm:mainresult}. 
The rest of this paper is devoted to its proof, which is structured as follows.

Section~\ref{sec:RegExt} develops, along the lines of \cite{Savare}, a global  regularity shift and provides localized interior regularity for the extension problem.
In Section~\ref{sec:LocTgReg},
we establish  
a local regularity shift for the tangential derivatives of the solution of the extension problem, 
in a vicinity of (smooth parts of) the boundary.
These estimates are combined in Section~\ref{sec:WghHpPolygon} with covering arguments
and scaling to establish the weighted analytic regularity.

Section~\ref{sec:Concl} provides a brief summary of our main results,
and outlines generalizations and applications of the present results.
\subsection{Notation}
\label{sec:Nota}
For open $\omega \subseteq \R^d$ 
and $t \in \N_0$, the spaces $H^t(\omega)$ are the classical Sobolev spaces of order $t$. For $t \in (0,1)$, 
fractional order Sobolev spaces are given in terms of the Aronstein-Slobodeckij seminorm 
$|\cdot|_{H^t(\omega)}$ and the full norm $\|\cdot\|_{H^t(\omega)}$ by 
\begin{align}
\label{eq:norm} 
|v|^2_{H^t(\omega)} = \int_{x \in \omega} \int_{z \in \omega} \frac{|v(x) - v(z)|^2}{\abs{x-z}^{d+2t}}
\,dz\,dx, 
\qquad \|v\|^2_{H^t(\omega)} = \|v\|^2_{L^2(\omega)} + |v|^2_{H^t(\omega)}, 
\end{align}
where we denote the Euclidean norm in $\R^d$ by $\abs{\;\cdot\;}$. 
For bounded Lipschitz domains $\Omega \subset \R^d$ and $t \in (0,1)$,  we additionally introduce  
\begin{align*}
\widetilde{H}^{t}(\Omega) \coloneqq \left\{u \in H^t(\R^d) \,: \, u\equiv 0 \; \text{on} \; \R^d \backslash \overline{\Omega} \right\}, 
\quad  \norm{v}_{\widetilde{H}^{t}(\Omega)}^2 \coloneqq \norm{v}_{H^t(\Omega)}^2 + \norm{v/r^t_{\partial\Omega}}_{L^2(\Omega)}^2,
\end{align*}
where $r_{\partial\Omega}(x)\coloneqq\operatorname{dist}(x,\partial\Omega)$ 
 denotes the Euclidean distance of a point
 $x \in \Omega$ from the boundary $\partial \Omega$.
On $\widetilde{H}^t(\Omega)$ we have,  
by combining \cite[Lemma~{1.3.2.6}]{Grisvard} and \cite[Proposition~2.3]{acosta-borthagaray17}, the estimate 
\begin{equation}
\label{eq:Htildet-vs-HtRd}
\forall u \in \widetilde{H}^t(\Omega) \colon \quad 
\|u\|_{\widetilde{H}^t(\Omega)} \leq C |u|_{H^t(\R^d)} 
\end{equation}
for some $C > 0$ depending
only on $t$ and $\Omega$. 
For $t \in (0,1)\backslash \{\frac 1 2\}$, the norms $\norm{\cdot}_{\widetilde{H}^{t}(\Omega)}$  and $\norm{\cdot}_{H^{t}(\Omega)}$ are equivalent on $\widetilde{H}^{t}(\Omega)$, 
see, e.g., \cite[Sec.~{1.4.4}]{Grisvard}.
Furthermore, for $t > 0$, the space $H^{-t}(\Omega)$ denotes the dual space of $\widetilde{H}^t(\Omega)$, 
and we write $\skp{\cdot,\cdot}_{L^2(\Omega)}$ 
for the duality pairing that extends the $L^2(\Omega)$-inner product.
\bigskip

We denote by $\Rpos$ the positive real numbers.
For subsets $\omega \subset \R^d$, we will use the notation $\omega^+\coloneqq \omega \times \Rpos$
and $\omega^{\theta}\coloneqq \omega \times (0,\theta)$ for $\theta > 0$.
For any multi index $\beta = (\beta_1,\dots,\beta_d)\in \mathbb{N}^d_0$, we
denote $\dbeta = \partial^{\beta_1}_{x_1}\cdots \partial^{\beta_d}_{x_d}$ and
$\betam= \sum_{i=1}^d\beta_i$.
We adhere to the convention that empty sums are null, i.e., $\sum_{j=a}^b c_j =0$ when $b<a$; this even applies to the case where the terms $c_j$ may not be defined. We also follow the standard convention $0^0 = 1$.
  \bigskip

Throughout this article, 
we use the notation $\lesssim$ to abbreviate $\leq$ up to a generic constant $C>0$
that does not depend on critical parameters in our analysis.

\section{Setting}
\label{sec:Set}
There are several different ways to define the fractional Laplacian $(-\Delta)^s$ for $s \in (0,1)$. 
A classical definition on the full space ${\mathbb R}^d$ 
is in terms of the Fourier transformation ${\mathcal F}$, i.e., 
$({\mathcal F} (-\Delta)^s u)(\xi) = |\xi|^{2s} ({\mathcal F} u)(\xi)$. 
Alternative, equivalent definitions of $(-\Delta)^s$ are, e.g., 
via spectral, semi-group, or operator theory,~\cite{Kwasnicki} or via singular integrals.

In the following, we consider the integral fractional Laplacian defined 
pointwise for sufficiently smooth functions $u$ as the principal value integral  
\begin{align}\label{eq:fracLaplaceDef}
(-\Delta)^su(x) \coloneqq C(d,s) \; \text{P.V.} \int_{\R^d}\frac{u(x)-u(z)}{\abs{x-z}^{d+2s}} \, dz \quad \text{with} \quad
C(d,s)\coloneqq - 2^{2s}\frac{\Gamma(s+d/2)}{\pi^{d/2}\Gamma(-s)},
\end{align}
where $\Gamma(\cdot)$ denotes the Gamma function. 
We investigate the fractional differential equation 
\begin{subequations}\label{eq:modelproblem}
\begin{align}
 (-\Delta)^su &= f \qquad \text{in}\, \Omega, \\
 u &= 0 \quad \quad\, \text{in}\, \Omega^c\coloneqq\R^d \backslash \overline{\Omega},
\end{align}
\end{subequations}
where $s \in (0,1)$ and $f \in H^{-s}(\Omega)$ is a given right-hand side. 
Equation (\ref{eq:modelproblem}) is understood in weak form: 
Find $u \in \widetilde{H}^s(\Omega)$ such that 
\begin{equation}
\label{eq:weakform}
a(u,v)\coloneqq \skp{(-\Delta)^s u,v}_{L^2(\R^d)} = \skp{f,v}_{L^2(\Omega)} 
\qquad \forall v \in \widetilde{H}^s(\Omega). 
\end{equation}
The bilinear form $a(\cdot,\cdot)$ has the alternative representation
\begin{equation}
a(u,v) = 
 \frac{C(d,s)}{2} \int\int_{\R^d\times\R^d} 
 \frac{(u(x)-u(z))(v(x)-v(z))}{\abs{x-z}^{d+2s}} \, dz \, dx 
\qquad \forall u,v \in \widetilde{H}^s(\Omega). 
\end{equation}

Existence and uniqueness of $u \in \widetilde{H}^s(\Omega)$ follow from
the Lax--Milgram Lemma for any $f \in H^{-s}(\Omega)$, upon the observation
that the bilinear form $a(\cdot,\cdot): \widetilde{H}^s(\Omega)\times \widetilde{H}^s(\Omega)\to \R$ 
is continuous and coercive.
\subsection{The Caffarelli-Silvestre extension}
\label{sec:CSExt}
A very influential interpretation of the fractional Laplacian is provided by the so-called 
\emph{Caffarelli-Silvestre extension}, due to \cite{CafSil07}.
It showed that the nonlocal operator $(-\Delta)^s$ 
can be be understood as a 
Dirichlet-to-Neumann map of a degenerate, \emph{local} elliptic PDE 
on a half space in $\R^{d+1}$. 
Throughout the following text, we let 
\begin{equation}
\label{eq:alpha}
\alpha \coloneqq 1-2s. 
\end{equation}
\subsubsection{Weighted spaces for the Caffarelli-Silvestre extension}
\label{sec:WghtSpcCS}
Throughout the text, we single out the last component of points in $\R^{d+1}$ by writing them as 
$(x,y)$ with $x  = (x_1,\ldots,x_d)\in \R^d$, $y \in \R$. We introduce, for open sets $D \subset \R^d \times \Rpos$, 
the weighted $L^2$-norm $\| \cdot \|_{L^2_\alpha(D)}$ via
\begin{equation}\label{eq:wgtL2}
 \norm{U}_{L^2_\alpha(D)}^2 \coloneqq \int_{(x,y) \in D} y^{\alpha} \abs{U(x,y)}^2 dx \, dy.
\end{equation}
We denote by $L^2_\alpha(D)$ the space of functions on $D$ that are
square-integrable with respect to the weight $y^\alpha$. We introduce 
$H^1_\alpha(D):= \{U \in L^2_\alpha(D)\,\colon\, \nabla U \in L^2_\alpha(D)\}$
as well as the Beppo-Levi space 
${\operatorname{BL}}^1_{\alpha}\coloneqq \{U \in L^2_{loc}(\R^{d} \times \Rpos)\,:\, 
\nabla U \in L^2_\alpha(\R^d \times \Rpos)\}$. 
For elements $U \in {\operatorname{BL}}^1_\alpha$, one can give meaning to their trace 
at $y = 0$, which is denoted $\operatorname{tr} U$. 
Recalling $\alpha = 1-2s$, one has in fact
$\operatorname{tr} U \in H^s_{loc}(\R^d)$ (see, e.g., \cite[Lem.~3.8]{KarMel19}). If 
$\operatorname{supp} \operatorname{tr} U \subset \overline{\Omega}$ for some bounded Lipschitz domain $\Omega$, then 
$\operatorname{tr} U \in \widetilde H^{s}(\Omega)$ and 
\begin{align}
\label{eq:L3.8-KarMel19}
 \|\operatorname{tr} U\|_{\widetilde{H}^s(\Omega)} \stackrel{(\ref{eq:Htildet-vs-HtRd})}{\lesssim} |\operatorname{tr} U|_{H^s(\R^d)} 
\stackrel{\text{\cite[Lem.~{3.8}]{KarMel19}}}{\lesssim} \norm{\nabla U}_{L^2_\alpha(\R^d \times \Rpos)}
\end{align}
with an implied constant depending on $s$ and $\Omega$.

\subsubsection{The Caffarelli-Silvestre extension}
\label{sec:CSExt-sub}
Given $u \in \widetilde{H}^s(\Omega)$, 
let $U = U(x,y)$ denote the minimum norm extension of $u$ to $\R^d \times \Rpos$, i.e., 
$U = \operatorname{arg min} 
\{\|\nabla U\|^2_{L^2_\alpha(\R^d \times \Rpos)}\,|\, 
   U \in {\operatorname{BL}}^1_\alpha,\, 
   \operatorname{tr} U = u \;\mbox{in}\; H^s(\R^d) \}$. 
The function $U$ is indeed unique in ${\operatorname{BL}}^1_{\alpha}$ 
(see, e.g., \cite[p.~{2900}]{KarMel19}). 

The Euler-Lagrange equations corresponding to this extension problem read
\begin{subequations}\label{eq:extension}
\begin{align}
\label{eq:extension-a}
 \operatorname*{div} (y^\alpha \nabla U) &= 0  \;\quad\quad\text{in} \; \R^d \times (0,\infty), \\ 
\label{eq:extension-b}
 U(\cdot,0) & = u  \,\quad\quad\text{in} \; \R^d.
\end{align}
\end{subequations}
Henceforth, when referring to solutions of (\ref{eq:extension}), 
we will additionally understand that $U \in {\operatorname{BL}}^1_\alpha$. 

The relevance of \eqref{eq:extension} is due to the fact that
the fractional Laplacian applied to $u\in \widetilde{H}^s(\Omega)$ 
can be recovered as distributional normal trace of 
the extension problem \cite[Section 3]{CafSil07}, \cite{caffarelli-stinga16}:
\begin{align}\label{eq:DtNoperator}
(-\Delta)^s u =  -d_s \lim_{y\rightarrow 0^+} y^\alpha \partial_y U(x,y),
\qquad 
d_s = 2^{2s-1}\Gamma(s)/\Gamma(1-s). 
\end{align}
\subsection{Main result: weighted analytic regularity for polygonal domains in $\R^2$}
\label{sec:WgtAnRegR2}
The following theorem is the main result of this article. 
It states that, provided the data $f$ is analytic in $\overline{\Omega}$,
we obtain analytic regularity for the solution $u$ of \eqref{eq:modelproblem} in 
a scale of weighted Sobolev spaces. 
In order to specify these weighted spaces, 
we need additional notation.  
\bigskip

Let $\Omega \subset \R^2$ be a bounded, polygonal Lipschitz domain 
with finitely many vertices and (straight) edges. (Connectedness of the boundary is not necessary in the following.)
We denote by $\mathcal{V}$ the set of vertices and 
by $\mathcal{E}$ the set of the (open) edges. 
For $\mathbf{v} \in \mathcal{V}$ and $\mathbf{e} \in \mathcal{E}$, we define the distance functions 
\begin{align*} 
r_{\mathbf{v}}(x)\coloneqq|x - \mathbf{v}|, 
\qquad 
r_{\mathbf{e}}(x)\coloneqq\inf_{y \in \mathbf{e}} |x - y|, 
\qquad 
\rho_{\mathbf{v} \mathbf{e}}(x)\coloneqq r_{\mathbf{e}}(x)/r_{\mathbf{v}}(x). 
\end{align*} 
For each vertex $\mathbf{v} \in \mathcal{V}$, we denote by $\mathcal{E}_{\mathbf{v}}\coloneqq \{\mathbf{e} \in \mathcal{E}\,:\, \mathbf{v} \in \overline{\mathbf{e}}\}$
the set of all edges that meet at $\mathbf{v}$. 
For any $\mathbf{e} \in \mathcal{E}$, 
we define $\mathcal{V}_{\mathbf{e}}\coloneqq \{\mathbf{v} \in \mathcal{V}\,:\, 
\mathbf{v} \in \overline{\mathbf{e}}\}$ as the set of endpoints of $\mathbf{e}$. 
For fixed, sufficiently small $\omegaeps > 0$ and for 
$\mathbf{v} \in \mathcal{V}$, $\mathbf{e} \in \mathcal{E}$, 
we define 
vertex, vertex-edge and edge neighborhoods by 
\begin{align*}
\omegac^{\omegaeps} &\coloneqq \{x \in \Omega\,:\, r_{\mathbf{v}}(x) < \omegaeps 
                                     \quad \wedge \quad \rho_{\mathbf{v}\mathbf{e}}(x) \geq \omegaeps
\quad \forall \mathbf{e} \in \mathcal{E}_{\mathbf{v}}\}, 
\\
\omegace^{\omegaeps} &\coloneqq \{x \in \Omega\,:\, r_{\mathbf{v}}(x) < \omegaeps 
                                     \quad \wedge \quad \rho_{\mathbf{v}\mathbf{e}}(x) < \omegaeps \},
\\
\omegae^{\omegaeps} &\coloneqq \{x \in \Omega\,:\, r_{\mathbf{v}}(x) \geq \omegaeps  
                                     \quad \wedge \quad  r_{\mathbf{e}}(x) < \omegaeps^2
\quad \forall \mathbf{v} \in \mathcal{V}_{\mathbf{e}}\}.
\end{align*}
Figure~\ref{fig:vertex-notation} 
illustrates this notation near a vertex $\mathbf{v} \in \mathcal{V}$ of the
polygon. Throughout the paper, we will assume that $\omegaeps$ is small enough so
that $\omegac^{\omegaeps} \cap \omegacprime^{\omegaeps} = \emptyset$ for all $\mathbf{v} \neq
\mathbf{v}'$, that $\omegae^{\omegaeps} \cap \omegaeprime^{\omegaeps} = \emptyset$  for all
$\mathbf{e}\neq \mathbf{e}'$ and $\omegace^{\omegaeps}\cap \omegaceprime^{\omegaeps} = \emptyset$ for
all $\mathbf{v}\neq \mathbf{v}'$ and all $\mathbf{e}\neq \mathbf{e}'$.
We will
drop the superscript $\omegaeps$ unless strictly necessary.

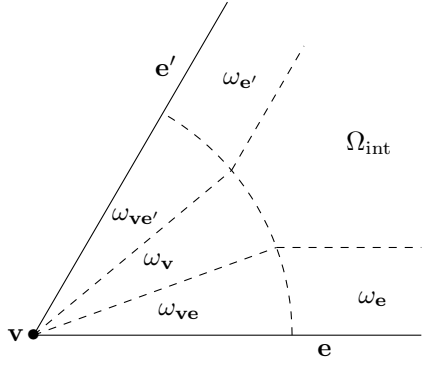
\begin{figure}[ht]
\begin{center}
\begin{tikzpicture}[scale=1.7]
  \def\R{3}
  \def\A{60}
\draw ({3/4*\R*cos(\A)-.1},{3/4*\R*sin(\A)}) node[above]{$\mathbf{e}'$};
\draw (0,0) node {\textbullet} node[left] {$\mathbf{v}$}; 
  \draw[-] (0, 0) -- ({\R*cos(\A)}, {\R*sin(\A)});
  \draw[-] (0, 0) -- (\R, 0); 
  \draw[dashed] (0, 0) -- ({2/3*\R*cos(\A/3)},  {2/3*\R*sin(\A/3)}); 
  \draw[dashed] (0, 0) -- ({2/3*\R*cos(2*\A/3)},{2/3*\R*sin(2*\A/3)}); 
  \draw[dashed] ({2/3*\R*cos(\A/3)}, {2/3*\R*sin(\A/3)}) -- (\R,{2/3*\R*sin(\A/3)}); 
  \draw[dashed] %
  ({2/3*\R*cos(2*\A/3)}, {2/3*\R*sin(2*\A/3)}) %
  -- ({2/3*\R*cos(2*\A/3) + \R*(1-2/3*cos(\A/3))*cos(\A)}, {2/3*\R*sin(2*\A/3) + \R*(1-2/3*cos(\A/3))*sin(\A)});
\draw (3/4*\R, 0) node [below]{$\mathbf{e}$} ;
\draw (3/8*\R ,0.05) node [above]{$\omega_{\mathbf{v}\mathbf{e}}$} ;
\draw (7/8*\R, 0.15) node [above]{$\omega_{\mathbf{e}}$} ;
\draw ({\R*cos(\A/2)}, {\R*sin(\A/2)}) node {$\Omega_{\mathrm{int}}$} ;
\draw ({3/8*\R*cos(\A/2)},{3/8*\R*sin(\A/2)-0.12}) node [above]{$\omega_{\mathbf{v}}$}; 
\draw ({3/8*\R*cos(3*\A/4)},{3/8*\R*sin(3*\A/4)}) node [above]{$\omega_{\mathbf{v}\mathbf{e}'}$} ;
\draw ({7/8*\R*cos(3*\A/4)-0.25},{7/8*\R*sin(3*\A/4)}) node [above]{$\omega_{\mathbf{e}'}$} ;
\draw [dashed,domain=0:\A] plot ({2/3*\R*cos(\x)}, {2/3*\R*sin(\x)});
\end{tikzpicture}
\end{center}
\caption{\label{fig:vertex-notation} Notation near a vertex $\mathbf{v}$.}
\end{figure}

We can decompose the Lipschitz polygon $\Omega$ into sectoral neighborhoods of vertices $\mathbf{v}$, 
which are unions of vertex and vertex-edge neighborhoods (as depicted in Figure~\ref{fig:vertex-notation}), 
edge neighborhoods (that are away from a vertex), and an interior part  $\Omega_{\rm int}$, i.e.,
\begin{align*}
 \Omega = \bigcup_{\mathbf{v} \in \mathcal{V}}\left( \omega_{\mathbf{v}} 
        \cup \bigcup_{\mathbf{e} \in \mathcal{E}_{\mathbf{v}}} \omega_{\mathbf{ve}} \right) 
        \cup \bigcup_{\mathbf{e} \in \mathcal{E}} \omega_{\mathbf{e}} 
        \cup \Omega_{\rm int}.
\end{align*}
Each sectoral and edge neighborhood may have a different value $\omegaeps$. 
However, since only finitely many different neighborhoods are needed to decompose the polygon, 
the interior part $\Omega_{\rm int}\subset\Omega $ has a positive distance from the boundary.  
\bigskip

In a given edge neighborhood $\omega_{\mathbf{e}}$ or an vertex-edge neighborhood 
$\omega_{\mathbf{v}\mathbf{e}}$, we let $\mathbf{e}_{\parallel}$ and $\mathbf{e}_{\perp}$ be two unit vectors 
such that $\mathbf{e}_{\parallel}$ is tangential to $\mathbf{e}$ and $\mathbf{e}_{\perp}$ is normal to $\mathbf{e}$. 
We introduce the differential operators 
\begin{align*}
D_{x_\parallel} v &\coloneqq \mathbf{e}_{\parallel} \cdot \nabla_x v, 
& D_{x_\perp} v &\coloneqq \mathbf{e}_{\perp} \cdot \nabla_x v
\end{align*}
corresponding to differentiation in the tangential and normal direction. 
Inductively, we can define higher order tangential and normal derivatives by 
$D_{x_\parallel}^j v \coloneqq D_{x_\parallel}(D_{x_\parallel}^{j-1} v)$ and
$D_{x_\perp}^j v \coloneqq D_{x_\perp}(D_{x_\perp}^{j-1} v)$ for $j>1$. \bigskip

Our main result provides local analytic regularity in edge- and vertex-weighted Sobolev spaces.
\begin{theorem}\label{thm:mainresult}
Let $\Omega \subset \R^2$ be a bounded polygonal Lipschitz domain.
  Let the data $f \in C^{\infty}(\overline{\Omega})$ satisfy
with a constant $\gamma_f>0$
  \begin{equation}
    \label{eq:analyticdata}
\forall j \in \N_0\colon \quad 
    \sum_{\betam = j} \|\dbeta f\|_{L^2(\Omega)} \leq \gamma_f^{j+1}j^j. 
  \end{equation}
Let $\mathbf{v} \in \mathcal{V}$, $\mathbf{e} \in \mathcal{E}$ and
$\omega_{\mathbf{v}}$, $\omega_{\mathbf{ve}}$, $\omega_{\mathbf{e}}$ be 
fixed vertex, vertex-edge and edge-neighborhoods.
Let $u$ be the weak solution of \eqref{eq:modelproblem}.
Then, there is $\gamma > 0$ depending only on $\gamma_f$, $s$, and $\Omega$ such that for 
every $\varepsilon>0$ there exists $C_\varepsilon>0$ (depending only on $\varepsilon$ and $\Omega$) 
such that the following holds: 
\begin{enumerate}[(i)]
\item 
For all $\beta \in \mathbb{N}^2_0$ there holds with $\abs{\beta}=p$
  \begin{equation}
 \label{eq:analytic-u-c-all}
 \norm{r_{\mathbf{v}}^{p-1/2-s+\varepsilon} \dbeta u }_{L^2(\omega_{\mathbf{v}})} 
 \leq C_{\varepsilon} \gamma^{p+1}p^p.
  \end{equation}
\item 
   For all $(\pperp, \ppar)\in \mathbb{N}^2_0$ there holds  with $p=\pperp+\ppar$
\begin{align}  
 \norm{r_{\mathbf{e}}^{\pperp-1/2-s+\varepsilon} D^{\pperp}_{x_\perp} D^{\ppar}_{x_{\parallel}} u }_{L^2(\omegae)} 
 & \leq C_{\varepsilon} \gamma^{p+1} p^p,
\label{eq:analytic-u-e-all}\\
  \norm{r_{\mathbf{e}} ^{\pperp-1/2-s + \varepsilon} r_{\mathbf{v}} ^{\ppar+\varepsilon} D^{\pperp}_{x_\perp} D^{\ppar}_{x_\parallel} u}_{L^2(\omegace)}
 & \leq C_{\varepsilon} \gamma^{p+1} p^p.
\label{eq:analytic-u-ce-all}
\end{align}
\item 
For the interior part $\Omega_{\rm int}$ and 
for all $\beta\in \N_0^2$ there holds with $\betam = p$
\begin{equation}
  \label{eq:analytic-u-int}
 \norm{\dbeta u }_{L^2(\Omega_{\rm int})} \leq \gamma^{p+1}p^p.
\end{equation}
\end{enumerate}
\end{theorem}

\begin{remark}
  Inequalities \eqref{eq:analytic-u-e-all} and \eqref{eq:analytic-u-ce-all} can
  be written in compact form: For all $\nu >-1/2-s$, there exists $C_\nu>0$ such that
  for $\bullet\in\{\mathbf{e}, \mathbf{ve}\}$ 
  \begin{equation}\label{eq:analytic-u-compact}
    \| r_{\mathbf{v}}^{p+\nu} \rho_{\mathbf{ve}}^{\pperp+\nu} D_{x_\parallel}^{\ppar} D_{x_\perp}^{\pperp} u\|_{L^2(\omega_\bullet)}
     \leq C_{\nu}\gamma^{p+1}p^p  \qquad \forall(\pperp, \ppar)\in \mathbb{N}^2_0\text{ with }p=\ppar+\pperp.
\end{equation}
\eremk
\end{remark}

\begin{remark}
\begin{enumerate}[(i)]
\item Stirling's formula implies $p^p \leq C p! e^p$. 
      Therefore, there exists a constant  $\widetilde{C}_\nu$ such that \eqref{eq:analytic-u-compact} can also be written as
\begin{equation}  \label{eq:analytic-u-ce1-factorial}
\| r_{\mathbf{v}}^{p+\nu} \rho_{\mathbf{ve}}^{\pperp+\nu} D_{x_\parallel}^{\ppar} D_{x_\perp}^{\pperp} u\|_{L^2(\omega_\bullet)}
\leq \widetilde{C}_{\nu} ({\gamma e)}^{p+1} p!,
\end{equation}
and the same can also be done for \eqref{eq:analytic-u-c-all} and \eqref{eq:analytic-u-int} in
Theorem~\ref{thm:mainresult}.

\item 
We note that $(\ppar+\pperp)^{\ppar+\pperp} \leq \ppar^{\ppar} \pperp^{\pperp} e^{\ppar+\pperp}$.
Together with $p^p \leq C p! e^p$ (using Stirling's formula), one can also formulate the estimates \eqref{eq:analytic-u-compact} as follows:  There are constants $\widetilde{C}_{\nu}$ and $\widetilde\gamma > 0$ such that 
\begin{align}  \label{eq:analytic-u-ce-factorial}
\forall (\ppar, \pperp )\in \N_0^2 \colon \quad 
    \| r_{\mathbf{v}}^{p+\nu} \rho_{\mathbf{ve}}^{\pperp+\nu} D_{x_\parallel}^{\ppar} D_{x_\perp}^{\pperp} u\|_{L^2(\omega_\bullet)} \leq \widetilde{C}_{\nu} \widetilde{\gamma}^{\pperp + \ppar} \pperp!  \;\ppar!.
\end{align}

\item 
The assumption  (\ref{eq:analyticdata}) 
on the data $f$ expresses analyticity in $\overline{\Omega}$ (combine Morrey's embedding \cite[eq.~(1,4,4,6)]{Grisvard} to see $f \in C^\infty$ with 
\cite[Lemma~{5.7.2}]{morrey66}).
Inspection of the proof
(in particular Lemmas~\ref{lemma:regularity-omega_c} and \ref{lemma:regularity-omega_e}) shows that $f$ could be admitted to be 
in vertex or edge-weighted classes of analytic functions.  For simplicity of exposition, we do
not explore this further. 
\item 
Inspection of the proofs also shows that, in order to obtain weighted
  regularity of fixed, finite order $p$, only finite regularity of the data $f$ is required.
\item 
By Morrey's embedding, e.g., \cite[eq.~(1,4,4,6)]{Grisvard}, 
estimate \eqref{eq:analytic-u-int} implies that the solution $u \in C^{\infty}(\overline{\Omega_{\rm int}})$ as well as analyticity of $u$ in 
$\overline{\Omega}_{\rm int}$, \cite[Lemma~{5.7.2}]{morrey66}.
Other results on interior analytic regularity of more general, linear 
integro-differential operators are, e.g., in \cite{GFV2015}, for $1/2<s<1$.
\eremk
\end{enumerate}
\end{remark}

\section{Regularity results for the extension problem}
\label{sec:RegExt}
In this section, we derive local (higher order) regularity results for solutions to the Caffarelli-Silvestre extension problem. 
As the techniques employed are valid in any space dimension, we formulate our results for general $d\in {\mathbb N}$. 

Fix $H > 0$. 
Given $F \in L^2_{-\alpha}(\R^d \times (0,H))$ and $f \in H^{-s}(\Omega)$,
consider the problem to
find the minimizer $U = U(x,y)$ with $x \in \R^d$ and $y \in \Rpos$ of 
\begin{align} \label{eq:minimization}
\mbox{ minimize ${\mathcal F}$ on 
${\operatorname{BL}}^1_{\alpha,0,\Omega}
\coloneqq  
\{U \in {\operatorname{BL}}^1_{\alpha}\,:\, \operatorname{tr} U =0 \; \mbox{ on $\Omega^c$}\},$ }
\end{align}
where 
\begin{align}
\label{eq:calFdef}
\mathcal{F}(U)& \coloneqq  \frac{1}{2} b(U,U) - \int_{\R^d \times (0,H)} F U \, dx\,dy - \int_{\Omega} f \operatorname{tr} U \,dx, & 
b(U,V) & := \int_{\R^d \times \Rpos} y^\alpha \nabla U \cdot \nabla V\,dx\,dy. 
\end{align}

We have the following Poincar\'e type estimate: 
\begin{lemma}
\label{lemma:properties-of-H1alpha}
\begin{enumerate}[(i)]
\item
\label{item:lemma:properties-of-H1alpha-i}
The map ${\operatorname{BL}}^1_{\alpha,0,\Omega} \ni U \mapsto \|\nabla U\|_{L^2_\alpha(\R^d \times \R_+)}$ is a norm, 
and 
${\operatorname{BL}}^1_{\alpha,0,\Omega}$ endowed with this norm is a Hilbert space
with corresponding inner-product given by the bilinear form $b(\cdot,\cdot)$ in \eqref{eq:calFdef}.
\item
\label{item:lemma:properties-of-H1alpha-ii}
For every $H\in (0,\infty)$, there is $C_{H,\alpha} > 0$ such that 
\begin{equation}
\label{eq:lemma:properties-of-H1alpha}
\forall U \in {\operatorname{BL}}^1_{\alpha,0,\Omega}\colon \quad 
\|U\|_{L^2_{\alpha}(\R^d \times (0,H))} \leq C_{H,\alpha} \|\nabla U\|_{L^2_\alpha(\R^d\times \R_+)}.  
\end{equation}
\end{enumerate}
\end{lemma}
\begin{proof}
Details of the proof
are given in Appendix~\ref{app:lemma:properties-of-H1alpha}. 
\end{proof}

With Lemma~\ref{lemma:properties-of-H1alpha} in hand, 
existence and uniqueness of solutions of (\ref{eq:minimization}) follows from 
the Lax-Milgram Lemma since, for $F \in L^2_{-\alpha}(\R^d \times (0,H))$ and $f \in H^{-s}(\Omega)$, 
the map $U \mapsto \int_{\R^d \times (0,H)} F U + \int_\Omega f \operatorname{tr} U$ 
in \eqref{eq:calFdef} extends to a bounded linear functional on ${\operatorname{BL}}^1_{\alpha,0,\Omega}$.
In view of (\ref{eq:lemma:properties-of-H1alpha}) and the trace estimate 
(\ref{eq:L3.8-KarMel19}),
the minimization problem (\ref{eq:minimization}) 
admits by Lax-Milgram a unique solution $U \in {\operatorname{BL}}^1_{\alpha,0,\Omega}$
with the {\sl a priori} estimate 
\begin{align}
\label{eq:energy-estimate} 
\|\nabla U\|_{L^2_{\alpha}(\R^d \times \Rpos)} \leq C \left[ \|F\|_{L^2_{-\alpha}(\R^d \times (0,H))} + \|f\|_{H^{-s}(\Omega)}\right] 
\end{align}
with constant $C$ dependent on $s\in (0,1)$, $H > 0$, and $\Omega$. 

The Euler-Lagrange equations formally satisfied by the solution $U$ of (\ref{eq:minimization}) are: 
\begin{subequations}
\label{eq:extension2D}
\begin{align}
\label{eq:extension2D-a}
 -\operatorname*{div} (y^\alpha \nabla U) &= F  &&\text{in} \; \R^d \times (0,\infty), \\ 
\label{eq:extension2D-b}
 \partial_{n_\alpha} U(\cdot,0) & = f  &&\text{in} \; \Omega, \\
\label{eq:extension2D-c}
\operatorname{tr} U & = 0 &&\text{on $\Omega^c$},
\end{align}
\end{subequations}
where $\partial_{n_\alpha} U(x,0) = - d_s\lim_{y \rightarrow 0}  y^\alpha \partial_y U(x,y)$ and we implicitly extended $F$ to $\R^d \times \Rpos$ by zero.
In view of \eqref{eq:DtNoperator} together with the fractional PDE $(-\Delta)^s u = f$, this is a Neumann-type Caffarelli-Silvestre extension problem with an additional source $F$.

\begin{remark}\label{remk:local-extension2D}
\begin{enumerate}[(i)]
\item \label{remk:local-extension2D-1}
The system \eqref{eq:extension2D} is understood in a weak sense, i.e., to find $U \in {\operatorname{BL}}^1_{\alpha,0,\Omega}$ such that 
\begin{equation}
\label{eq:weak-form}
\forall V \in {\operatorname{BL}}^1_{\alpha,0,\Omega} \colon \quad 
b(U,V) = \int_{\R^d \times \Rpos}  F V \,dx\,dy + \int_\Omega f \operatorname{tr} V\,dx. 
\end{equation}
Due to \eqref{eq:lemma:properties-of-H1alpha}, the integral $\int_{\R^d \times \Rpos}  F V \,dx\,dy $ is well-defined.

\item \label{remk:local-extension2D-2}
For the notion of solution of \eqref{eq:extension2D}, the support requirement $\operatorname{supp} F \subset \R^d \times [0,H]$ can be relaxed 
e.g., to $F \in L^2_{-\alpha}(\R^d \times \Rpos)$ 
by testing with $V \in H^1_{\alpha,0,\Omega}(\R^d \times \Rpos):= H^1_\alpha(\R^d \times \Rpos) \cap {\operatorname{BL}}^1_{\alpha,0,\Omega} $. In this case, the integral $\int_{\R^d \times \Rpos}  F V \,dx\,dy $ is well-defined by Cauchy-Schwarz.

\item \label{remk:local-extension2D-4}
We note that working with functions supported in $\R^d \times [0,H]$ 
induces an implicit dependence on $H$ of all constants, 
which is due to the Poincar\'e type estimate \eqref{eq:lemma:properties-of-H1alpha}. 
Alternatively to restricting the test space, 
one could also circumvent this by introducing suitable weights 
that control the behavior of $F$ at infinity; we do not develop this here.
\eremk
\end{enumerate}
\end{remark}

\subsection{Global regularity: a shift theorem}
\label{sec:GlRegShThm}
The following lemma provides additional regularity of the extension problem in the $x$--direction. 
The argument uses the technique developed in \cite{Savare} (see also \cite{EF1999,Eb2002})
that has recently been used in \cite{BN21} 
to show a closely related shift theorem for the Dirichlet fractional Laplacian; 
the technique merely assumes $\Omega$ to be a Lipschitz domain in $\R^d$. 
On a technical level, 
the difference between \cite{BN21} and Lemma~\ref{lem:regularity2D} below is that 
Lemma~\ref{lem:regularity2D} studies (tangential) differentiability
properties of the extension $U$, whereas \cite{BN21} focuses on the trace $u = \operatorname{tr} U$.
\bigskip

For functions $U$, $F$, $f$, it is convenient to introduce the abbreviation
\begin{equation}
\label{eq:CUFf} 
N^2(U,F,f)\coloneqq  \|\nabla U\|_{L^2_\alpha(\R^d \times \Rpos)} 
\left(
\|\nabla U\|_{L^2_{\alpha}(\R^d \times \Rpos)} 
+ 
\|F\|_{L^2_{-\alpha}(\R^d \times (0,H))}
+ 
\|f\|_{H^{1-s}(\Omega)}\right). 
\end{equation}
In view of the {\sl a priori} estimate \eqref{eq:energy-estimate}, 
we have the simplified bound (with updated constant $C$)
\begin{equation}
\label{eq:CUFf-simplified}
N^2(U,F,f) 
\leq 
C \left( \|f\|^2_{H^{1-s}(\Omega)} + \|F\|^2_{L^2_{-\alpha}(\R^d\times(0,H))} \right).
\end{equation}

\begin{lemma}
\label{lem:regularity2D}
Let $\Omega \subset \R^d$ be a bounded Lipschitz domain, and
let $B_{\widetilde R} \subset \R^d$ be a ball with $\Omega \subset B_{\widetilde R}$. 
For $t \in [0,1/2)$, there is a constant 
$C_t > 0$ (depending only on $s$, $t$, $\Omega$, $\widetilde R$, and $H$) 
such that for $f\in C^\infty(\overline{\Omega})$,
$F \in L^2_{-\alpha}(\R^d \times (0,H))$ 
the solution $U$ of (\ref{eq:minimization}) satisfies 
\begin{align*}
\int_{\Rpos} y^\alpha \norm{\nabla U(\cdot, y)}_{H^t(B_{\widetilde R})}^2 dy 
\leq C_t N^2(U,F,f)
\end{align*}
with $N^2(U,F,f)$ given by (\ref{eq:CUFf}). 
\end{lemma}
\begin{proof}
The idea is to apply the difference quotient argument from \cite{Savare} only in the $x$-direction.

Let $x_0 \in \overline{\Omega}$ be arbitrary. For $h \in \R^d$ denote 
$T_h U \coloneqq  \eta U_h + (1-\eta)U$, where $U_h(x,y) \coloneqq  U(x+h,y)$ and 
$\eta$ is a cut-off function that localizes to a suitable ball $B_{2\rho}(x_0)$, i.e, $0\leq \eta \leq 1$, 
$\eta \equiv 1$ on  $B_{\rho}(x_0)$ and $\operatorname{supp}\eta \subset B_{2\rho}(x_0)$. In Steps~1--5 of this proof, 
we will abbreviate $B_{\rho'}$ for $B_{\rho'}(x_0)$ for $\rho' > 0$. 

The main result of \cite{Savare} is that estimates for the modulus $\omega(U)$ defined 
with the quadratic functional $\mathcal{F}$ as in \eqref{eq:calFdef} by 
\begin{align*}
\omega(U) &\coloneqq  \sup_{h \in D\backslash \{0\}} \frac{\mathcal{F}(T_h U) - \mathcal{F}(U)}{\abs{h}}  = \omega_{b}(U) + \omega_{F}(U) + \omega_{f}(U), \\
\omega_b(U) &\coloneqq   \frac{1}{2}\sup_{h \in D\backslash  \{0\}} \frac{b(T_h U,T_h U) - b(U,U)}{\abs{h}} , \\ 
\omega_F(U) &\coloneqq   \sup_{h \in D\backslash  \{0\}} \frac{\int_{\R^d \times (0,H)} F (T_h U - U)}{\abs{h}}, 
\qquad 
\omega_f(U) \coloneqq   \sup_{h \in D\backslash  \{0\}} \frac{\int_{\Omega}  f \operatorname{tr}( T_h U - U)}{\abs{h}}, 
\end{align*}
can be used to derive regularity results in Besov spaces.
Here, $D \subset \R^d$ denotes a set of admissible directions $h$. These directions are chosen such that the function $T_h U$ is an 
admissible test function, i.e., $T_h U \in {\operatorname{BL}}^1_{\alpha,0,\Omega}$. For this, we have to require 
$\supp \operatorname{tr}(T_h U) \subset \overline{\Omega}$.
In \cite[(30)]{Savare} a description of this set is given in terms of a set of admissible outward pointing vectors 
$\mathcal{O}_\rho(x_0)$, which are directions 
$h$ with $\abs{h} \leq \rho$ such that for all $t \in [0,1]$ the translate $B_{3\rho}(x_0)\backslash \Omega + th$ is completely contained in 
$\Omega^c$.

{\bf Step 1.} (Estimate of $\omega_b(U)$). The definition of the bilinear form $b(\cdot,\cdot)$, $h \in D$,  
and the definition of $T_h$ give 
\begin{align*}
& b(T_h U,T_h U)-b(U,U) = \int_{\R^d \times \Rpos}y^\alpha ( \abs{\nabla T_h U}^2- \abs{\nabla U}^2) \, dx \, dy \\
& \qquad =  \int_{\R^d \times \Rpos}y^\alpha ( \abs{(\nabla \eta) (U_h-U) + T_h \nabla U}^2- \abs{\nabla U}^2) \, dx \, dy \\
&\qquad = \int_{\R^d \times \Rpos}y^\alpha ( \abs{(\nabla \eta) (U_h-U)}^2 + 2 T_h \nabla U \cdot (\nabla \eta) (U_h-U)) \, dx \, dy 
+  
\int_{\R^d \times \Rpos}y^\alpha(\abs{T_h \nabla U}^2- \abs{\nabla U}^2) \, dx \, dy \\
&\qquad \eqqcolon  T_1 + T_2.
\end{align*}
For the first integral $T_1$, 
we use the support properties of $\eta$ and that 
$\norm{U(\cdot,y)-U_h(\cdot,y)}_{L^2(B_{2\rho})} \lesssim \abs{h} \norm{\nabla U(\cdot,y)}_{L^2(B_{3\rho})}$, 
which gives
\begin{align*}
T_1 &\lesssim  \int_{\Rpos}y^\alpha(\abs{h}^2\norm{\nabla U(\cdot,y)}_{L^2(B_{3\rho})}^2  + \abs{h}\norm{\nabla U(\cdot,y)}_{L^2(B_{3\rho})}\norm{T_h \nabla U(\cdot,y)}_{L^2(B_{2\rho})}) \, dy  \\
& \lesssim \abs{h} \int_{B_{3\rho}^+}y^\alpha \abs{\nabla U}^2\, dx \, dy.
\end{align*}
For the term $T_2$, we first note $\abs{T_h \nabla U}^2 \leq \eta \abs{\nabla U_h}^2 + (1-\eta)\abs{\nabla U}^2$ 
since $0 \leq \eta \leq 1$. Using the variable transformation $z = x+h$ together with $B_{2\rho}(x_0)+h \subset B_{3\rho}(x_0)$ we obtain
\begin{align*}
T_2 &= \int_{\R^d \times \Rpos}y^\alpha(\abs{T_h \nabla U}^2- \abs{\nabla U}^2) \, dx \, dy
\leq 
 \int_{\Rpos}\int_{B_{2\rho}}y^\alpha \eta (\abs{\nabla U_h}^2- \abs{\nabla U}^2) \, dx \, dy \\
 &\leq   \int_{\Rpos}\int_{B_{3\rho}}y^\alpha (\eta(x-h)-\eta(x)) \abs{\nabla U}^2 \, dx \, dy \lesssim 
 \abs{h} \int_{B_{3\rho}^+}y^\alpha \abs{\nabla U}^2 \, dx \, dy.
\end{align*}
Altogether we get from the previous estimates that
$$
\omega_b(U) \lesssim \int_{B_{3\rho}^+}y^\alpha \abs{\nabla U}^2 \, dx \, dy.
$$

{\bf Step 2.} (Estimate of $\omega_F(U)$).
Using the definition of $T_h$, we can write $U-T_h U = \eta(U-U_h)$, and 
$\operatorname{supp}\eta \subset B_{2\rho}(x_0)$ implies 
\begin{align}\label{eq:SavTmp2}
 \left| \int_{\R^d \times (0,H)} F (U-T_h U)\, dx \, dy\right| &= \left| \int_{\R^d \times (0,H)} F\eta (U-U_h) \, dx \, dy \right| \leq \norm{F}_{L^2_{-\alpha}(B_{2\rho}\times (0,H))}
 \norm{U-U_h}_{L^2_\alpha(B_{2\rho}^+)} \nonumber \\
 &\lesssim \abs{h} \norm{F}_{L^2_{-\alpha}(B_{2\rho} \times (0,H))} \norm{\nabla U}_{L^2_{\alpha}(B^+_{3\rho})},
\end{align}
which produces 
$$
\omega_F(U) \lesssim \|F\|_{L^2_{-\alpha}(B_{2\rho}\times (0,H))} \|\nabla U\|_{L^2_{\alpha}(B^+_{3\rho})}. 
$$

{\bf Step 3.} (Estimate of $\omega_f(U)$).
For the trace term, we use a second cut-off function $\widetilde\eta \in C^\infty_0({\mathbb R}^{d+1})$ with $\widetilde \eta \equiv 1$ 
on $B_{3\rho}(x_0) \times \{0\} $ and $\operatorname{supp}(\widetilde \eta) \subset B_{4\rho}(x_0) \times (-H,H)$ and  get with the trace inequality (\ref{eq:L3.8-KarMel19}) and the estimate (\ref{eq:lemma:properties-of-H1alpha})
since $\operatorname{supp} (f \eta - (f\eta)_{-h}) \subset B_{3\rho}$
\begin{align}\label{eq:SavTmp3}
\left|  \int_{\Omega} f \operatorname{tr}(U-T_h U)\, dx \right| &= 
 \left| \int_{B_{2\rho}} f \eta \operatorname{tr}(U-U_h)\, dx \right| =  
\left| \int_{B_{3\rho}} (f \eta - (f\eta)_{-h}) \operatorname{tr}(\widetilde \eta U)\, dx\right|  \nonumber
\\&\leq \norm{f \eta - (f\eta)_{-h}}_{H^{-s}(B_{3\rho})} \norm{\operatorname{tr}(\widetilde \eta U)}_{\widetilde{H}^s(B_{3\rho})} \nonumber
\\ 
&\stackrel{(\ref{eq:L3.8-KarMel19}), (\ref{eq:lemma:properties-of-H1alpha})}{\lesssim} \abs{h} \norm{f}_{H^{1-s}(B_{4\rho})}
\|\nabla U\|_{L^2_\alpha(\R^d \times \Rpos)},
 \end{align}
where the estimate $\|f \eta - (f \eta)_{-h}\|_{H^{-s}(B_{3 \rho})} \lesssim |h| \|f\|_{H^{1-s}(B_{4\rho})}$ can be seen, for example, 
by interpolating the estimates
$\|f \eta - (f \eta)_{-h}\|_{H^{-1}(\R^d)} \lesssim |h| \|\eta f\|_{L^{2}(\R^d)}$ and 
$\|f \eta - (f \eta)_{-h}\|_{L^{2}(\R^d)} \lesssim |h| \|\eta f\|_{H^{1}(\R^d)}$, see, e.g.,  \cite{tartar07}.  We have thus obtained
\begin{align*}
\omega_f(U) \lesssim \|f\|_{H^{1-s}(B_{4\rho})} 
\|\nabla U\|_{L^2_\alpha(\R^d \times \Rpos)}.  
\end{align*}

{\bf Step 4.} (Application of the abstract framework of \cite{Savare}).
We introduce the seminorm $[U]^2\coloneqq  \int_{\R^d \times \Rpos} y^\alpha |\nabla U|^2\,dx dy$. By the 
coercivity of $b(\cdot,\cdot)$ on ${\operatorname{BL}}^1_{\alpha,0,\Omega}$ with respect to $[\cdot]^2$ 
and the abstract 
estimates in \cite[Sec.~{2}]{Savare}, we have 
\begin{align*}
[U-T_h U]^2  \overset{\footnotesize\cite{Savare}}&{\lesssim}  \omega(U) |h| \lesssim |h| \left(\omega_b(U) + \omega_F(U) + \omega_f(U)\right) \\
 \overset{\text{steps~1-3}}&{\leq }
|h|\left(
\|\nabla U\|^2_{L^2_\alpha(B^+_{3\rho})} 
+ \|F\|_{L^2_{-\alpha}(B^+_{2\rho})} \|\nabla U\|_{L^2_\alpha(\R^d \times \Rpos)} + \|f\|_{H^{1-s}(B_{4\rho})}\|\nabla U\|_{L^2_\alpha(\R^d \times \Rpos)} 
\right) \\
& \eqqcolon |h| \; \widetilde{C}^2_{U,F,f}. 
\end{align*}
Using that $\eta \equiv 1$ on $B^+_{\rho}(x_0)$, we get 
\begin{align} 
\label{eq:local-10} 
\int_{B^+_\rho} y^\alpha |\nabla U - \nabla U_h|^2\, dx \, dy &\leq
\int_{\R^d \times \Rpos} y^\alpha |\nabla (\eta U - \eta U_h)|^2\, dx \, dy  
 = 
[U - T_h U]^2 \leq |h| \; \widetilde C^2_{U,F,f}.
\end{align} 

{\bf Step 5:} (Removing the restriction $h \in D$). 
The set $D$ contains a truncated cone $C = \{x \in \R^d\,:\, |x\cdot e_D| > \delta |x|\} \cap B_{R'}(0)$ for some unit vector $e_D$ 
and $\delta  \in (0,1)$, $R' > 0$. 
Geometric considerations then show that there is $c_D > 0$ sufficiently large such that for arbitrary $h \in \R^d$ sufficiently small, 
$h + c_D |h| e_D \in D$. For a function $v$ defined on $\R^d$, we observe 
\begin{align*}
v(x) - v_h(x) & = v(x) - v(x+h) 
                   = v(x) - v(x+(h+c_D |h| e_D)) + v((x+h) + c_D |h| e_D) - v(x+h). 
\end{align*}
We may integrate over $B_{\rho'}(x_0)$ and  change variables to get 
\begin{align*}
\norm{v - v_h}^2_{L^2(B_{\rho'})} \leq 
2 \norm{v - v_{h+c_D |h| e_D}}^2_{L^2(B_{\rho'})} + 2 \norm{v - v_{c_D |h| e_D}}^2_{L^2(B_{\rho'} + h)}.
\end{align*}
Selecting $\rho' = \rho/2$ and for $|h| \leq \rho/2$, we obtain 
\begin{align*}
\norm{v  - v_h}^2_{L^2(B_{\rho/2})} \leq 2 \norm{v - v_{h+c_D|h| e_D}}^2_{L^2(B_{\rho})} + 
2 \norm{v - v_{c_D |h| e_D}}^2_{L^2(B_\rho)}. 
\end{align*}
Applying this estimate with $v = \nabla U$ and using that $h + c_D |h| e_D \in D$ and $c_D |h| e_D \in D$, we get from (\ref{eq:local-10}) that
\begin{align*}
\norm{\nabla U - \nabla U_h}^2_{L^2_\alpha(B^+_{\rho/2})} 
\lesssim |h| \; \widetilde C^2_{U,F,f}.
\end{align*}
The fact that $\Omega$ is a Lipschitz domain implies that the value of $\rho$ and the constants appearing in the definition of the truncated cone $C$ 
can be controlled uniformly in $x_0 \in \Omega$. Hence, covering the ball $B_{2\widetilde{R}}$ (with twice the radius as the ball $B_{\widetilde{R}}$) 
by finitely many balls $B_{\rho/2}$, we obtain with the constant $N(U,F,f)$ of (\ref{eq:CUFf}) 
that for all $h \in \R^d$ with $|h| \leq \delta' $ for some fixed $\delta' > 0$: 
\begin{align}
\label{eq:control-of-U-Uh}
\norm{\nabla U - \nabla U_h}^2_{L^2_\alpha(B_{2\widetilde{R}})}  
\lesssim |h| \; N^2(U,F,f). 
\end{align}

{\bf Step 6:} ($H^t(B_{\widetilde{R}})$--estimate). For $t < 1/2$, we estimate with the Aronstein-Slobodecki seminorm 
\begin{align*}
\int_{\Rpos} |\nabla U(\cdot,y)|^2_{H^t(B_{\widetilde{R}})}\,dy  & \leq 
\int_{\Rpos} \int_{x \in B_{\widetilde{R}}} \int_{|h| \leq {2\widetilde{R}}} \frac{|\nabla U(x+h,y) - \nabla U(x,y)|^2}{|h|^{d+2t}}\,dh\; dx\; dy.  
\end{align*}
The integral in $h$ is split into the range $|h| \leq \varepsilon$ for some fixed $\varepsilon>0$, for which 
(\ref{eq:control-of-U-Uh}) can be brought to bear, and $\varepsilon < |h| < 2\widetilde{R}$, for which a triangle inequality can be used. We obtain 
\begin{align*}
\int_{\Rpos} |\nabla U(\cdot,y)|^2_{H^t(B_{\widetilde{R}})}\,dy  & \lesssim 
N^2(U,F,f) \int_{|h| \leq \varepsilon} |h|^{1-d-2t}\,dh + \|\nabla U\|^2_{L^2_\alpha(\R^d \times \Rpos)} \int_{\varepsilon < |h| < 2\widetilde{R}} |h|^{-d-2t}\,dh
\\ &\lesssim N^2(U,F,f),
\end{align*}
which is the sought estimate. 
\end{proof}

\begin{remark}
The regularity assumptions on $F$ and $f$ can be weakened by interpolation techniques as described in \cite[Sec.~4]{Savare}. 
For example, by linearity, we may write $U = U_F + U_f$, where $U_F$ and $U_f$ solve
(\ref{eq:extension2D}) for data $(F,0)$ and $(0,f)$. The {\sl a priori} estimate (\ref{eq:energy-estimate}) gives 
$\|\nabla U_f\|_{L^2_{\alpha}(\R^d \times \Rpos)} \leq C \|f\|_{H^{-s}(\Omega)}$ so that we have 
\begin{align*}
\int_{\Rpos} |\nabla U_f(\cdot,y)|^2_{H^t(B_{\widetilde{R}})}\,dy & \leq C_t \left( \|\nabla U_f\|^2_{L^2_\alpha(\R^d \times \Rpos)} + \|f\|_{H^{1-s}(\Omega)} \|\nabla U_f\|_{L^2_\alpha(\R^d \times\Rpos)}\right) 
\\
& \lesssim \|f\|^2_{H^{-s}(\Omega)} + \|f\|_{H^{1-s}(\Omega)} \|f\|_{H^{-s}(\Omega)} \lesssim 
 \|f\|_{H^{1-s}(\Omega)} \|f\|_{H^{-s}(\Omega)}.  
\end{align*}
By, e.g., \cite[Lemma~{25.3}]{tartar07}, the mapping $f \mapsto U_f$ then satisfies 
\begin{align*}
\int_{\Rpos} |\nabla U_f(\cdot,y)|^2_{H^t(B_{\widetilde{R}})} \,dy
\leq C_t \|f\|^2_{B^{1/2-s}_{2,1}(\Omega)},
\end{align*}
where $B^{1/2-s}_{2,1}(\Omega) = (H^{-s}(\Omega),H^{1-s}(\Omega))_{1/2,1}$ is an interpolation space ($K$-method). 
We mention that $B^{1/2-s}_{2,1}(\Omega) \subset H^{1/2-s-\varepsilon}(\Omega)$ 
for every $\varepsilon > 0$. 
A similar estimate could, in principle,  be obtained for $U_F$; however,  the pertinent interpolation space is less tractable. 
\eremk
\end{remark}

\subsection{Interior regularity for the extension problem} 
\label{sec:IntReg}

In the following, we derive localized interior regularity estimates, also called Caccioppoli inequalities, for solutions to the extension problem \eqref{eq:extension2D}, 
where second order derivatives on some ball $B_R(x_0)\subset \Omega$ can be controlled by first order derivatives on some ball with a (slightly) larger radius.

The following Caccioppoli type inequality provides local control of higher order $x$-derivatives and is structurally similar to \cite[Lem.~4.4]{FMP21}. 
\begin{lemma}[Interior Caccioppoli inequality]
\label{lem:CaccType2D} 
Let $B_R \coloneqq  B_R(x_0) \subset \Omega \subset \R^d$ be an open ball of radius $R>0$ centered at $x_0 \in \Omega$, 
and let $B_{cR}$ be the concentric scaled ball of radius $cR$ with $c \in (0,1)$. Let $\zeta \in C^\infty_0({B_R})$ with 
$0 \leq \zeta \leq 1$ and $\zeta \equiv 1$ on $B_{cR}$ as well as 
$\|\nabla \zeta\|_{L^\infty(B_R)} \leq C_\zeta ((1-c)R)^{-1}$ for some $C_\zeta > 0$ 
independent of $c$, $R$.  
Let $U$ satisfy \eqref{eq:extension2D}
for given data $f$ and $F$ with $\supp F \subset \R^d \times [0,H]$.

Then, there exists a constant $C_{\rm int} > 0$, which depends only on $s$, $\Omega$, and $C_\zeta$,
such that for $i \in \{1,\dots,d\}$ 
\begin{align}
\label{eq:CaccType2D} 
 \norm{\partial_{x_i}(\nabla U)}^{2}_{L^2_\alpha(B_{cR}^+)} \leq \CacInt^{2} \left( ((1-c)R)^{-2} \norm{\nabla U}_{L^2_\alpha(B_R^+)}^{2}
 + \norm{\zeta \partial_{x_i} f}^{2}_{H^{-s}(\Omega)} + \norm{F}^{2}_{L^2_{-\alpha}(B_R^+)}\right). 
\end{align}
Furthermore, $\|\zeta \partial_{x_i} f\|_{H^{-s}(\Omega)} \leq C_{\rm loc} \|\partial_{x_i} f\|_{L^2(B_R)}$ for some $C_{\rm loc} > 0$ 
independent of $R$, $c$, and $f$ (cf.\ Lemma~\ref{lemma:localization-fractional-norms}).
\end{lemma}
\begin{proof}
The function $\zeta$ is defined on $\R^d$; through the constant extension we will also view it as a function on $\R^d \times \Rpos$.
With the unit vector $e_{x_i}$ in the $x_i$-coordinate and $\tau\in \R \backslash\{0\}$, we define 
the difference quotient 
\begin{align*}
D_{x_i}^\tau w(x) \coloneqq  \frac{w(x+\tau e_{x_i})-w(x)}{\tau}.
\end{align*}
For $|\tau|$ sufficiently small, we may use the test function $V = D_{x_i}^{-\tau}(\zeta^2 D_{x_i}^\tau U)$ 
in the weak formulation of \eqref{eq:extension2D}
(observe that this is an admissible test function and has support in $\overline{B_R^+}$) and compute
\begin{align*}
\operatorname{tr} V = -\frac{1}{\tau^2} \Big(\zeta^2(x-\tau e_{x_i})(u(x)-u(x-\tau e_{x_i}))+\zeta^2(x)(u(x)-u(x+\tau e_{x_i}))\Big) = D_{x_i}^{-\tau}(\zeta^2 D_{x_i}^\tau u).
\end{align*}
Integration by parts in \eqref{eq:extension2D} over $\R^d \times \Rpos$ and using that the Neumann trace (up to the constant $d_s$ from \eqref{eq:DtNoperator})
produces the fractional Laplacian gives
\begin{align*}
&\int_{\R^d \times \Rpos} F V \, dx \, dy - \frac{1}{d_s} \int_{\R^d} (-\Delta)^s u  \operatorname{tr}V \, dx = 
\int_{\R^d \times \Rpos} y^{\alpha} \nabla U \cdot\nabla V  dx\, dy \\ 
&\qquad\qquad\qquad= \int_{\R^d \times \Rpos} D_{x_i}^\tau (y^\alpha \nabla U) \cdot \nabla (\zeta^2 D_{x_i}^\tau U) \,dx\, dy \\ 
&\qquad\qquad\qquad=  \int_{B_R^+} y^\alpha D_{x_i}^\tau (\nabla U) \cdot \left(\zeta^2 \nabla D_{x_i}^\tau U + 2\zeta \nabla \zeta D_{x_i}^\tau U\right) dx \, dy\\
&\qquad\qquad\qquad =
  \int_{B_R^+} y^\alpha \zeta^2  D_{x_i}^\tau (\nabla U) \cdot D_{x_i}^\tau (\nabla U)\, dx\,dy + 
  \int_{B_R^+} 2 y^\alpha\zeta \nabla\zeta \cdot D_{x_i}^\tau (\nabla U)  D_{x_i}^\tau U\, dx\, dy. 
\end{align*}
We recall that by, e.g., \cite[Sec.~{6.3}]{evans98}, we have uniformly in $\tau$
\begin{equation}
\|D^\tau _{x_i} v\|_{L^2_{\alpha} (\R^d \times \Rpos)} \lesssim \|\partial_{x_i} v\|_{L^2_{\alpha}(\R^d\times\Rpos)}.
\end{equation}
Using the equation $(-\Delta)^s u = f$ on $\Omega$, Young's inequality, 
and the Poincar\'e inequality together with the trace estimate \eqref{eq:L3.8-KarMel19}, 
we get the existence of constants $C_j>0$, $j \in \{1,\dots,5\}$, such that
\begin{align*}
  \norm{\zeta D_{x_i}^\tau(\nabla U)}_{L^2_\alpha(B_R^+)}^2
& \leq 
    C_1 \bigg(
\abs{  \int_{B_R^+} y^\alpha \zeta \nabla\zeta \cdot  D_{x_i}^\tau (\nabla U) D_{x_i}^\tau U \, dx \, dy} + 
  \abs{\int_{\R^d \times \Rpos} F \; D_{x_i}^{-\tau} \zeta^2 D_{x_i}^\tau U \, dx \, dy}  \\
  & \qquad+ \abs{\int_{\R^d} D_{x_i}^{\tau} f  \zeta^2 D_{x_i}^\tau u \, dx}\bigg)\\
  &\leq
  \frac{1}{4}  \norm{\zeta D_{x_i}^\tau(\nabla U)}_{L^2_\alpha(B_R^+)}^2 +  
  C_2 \bigg(\norm{\nabla \zeta}_{L^{\infty}(B_R)}^2\norm{D_{x_i}^\tau U}_{L^2_\alpha(B_R^+)}^2 \\
 & \qquad +  \norm{F}_{L^2_{-\alpha}(B_R^+)} \|\partial_{x_i} (\zeta^2 D^\tau_{x_i} U)\|_{L^2_{\alpha}(B^+_R)} 
 + \norm{\zeta D_{x_i}^\tau f}_{H^{-s}(\Omega)} \norm{\zeta D^\tau_{x_i} u}_{H^s(\R^d)}  \bigg)\\
  & \leq
\frac{1}{2}  \norm{\zeta D_{x_i}^\tau(\nabla U)}_{L^2_\alpha(B_R^+)}^2 
+  
C_3 \bigg(\|\nabla \zeta\|_{L^\infty(B_R)}^2 \|\nabla U\|_{L^2_{\alpha}(B^+_R)}^2 + \|F\|^2_{L^2_{-\alpha}(B^+_R)} 
 \\
&\qquad\qquad + \norm{\zeta D_{x_i}^\tau f}_{H^{-s}(\Omega)} \abs{\zeta D^\tau_{x_i} u}_{H^s(\R^d)}\bigg) 
\\
  \overset{\eqref{eq:L3.8-KarMel19}}
  & \leq
\frac{1}{2}  \norm{\zeta D_{x_i}^\tau(\nabla U)}_{L^2_\alpha(B_R^+)}^2 
  +  
C_4 \bigg(\|\nabla \zeta\|_{L^\infty(B_R)}^2 \|\nabla U\|_{L^2_{\alpha}(B^+_R)}^2 + \|F\|^2_{L^2_{-\alpha}(B^+_R)}  
\\
&\qquad\qquad + \norm{\zeta D_{x_i}^\tau f}_{H^{-s}(\Omega)} 
   \norm{\nabla (\zeta D^\tau_{x_i} U)}_{L^2_{\alpha}( \R^d \times \Rpos)} \bigg)
\\
 &\leq 
   \frac{3}{4}  \norm{\zeta D_{x_i}^\tau(\nabla U)}_{L^2_\alpha(B_R^+)}^2
   \\ & \qquad \qquad
+  C_5 \bigg(\|\nabla \zeta\|_{L^\infty(B_R)}^2 \|\nabla U\|_{L^2_{\alpha}(B^+_R)}^2 + \|F\|^2_{L^2_{-\alpha}(B^+_R)} 
 + \norm{\zeta D_{x_i}^\tau f}^2_{H^{-s}(\Omega)}  \bigg).  
\end{align*}
Absorbing the first term of the right-hand side in the left-hand side and taking the limit $\tau \rightarrow 0$, 
we obtain the sought inequality for the second derivatives 
since $\norm{\nabla \zeta}_{L^{\infty}(B_R)} \lesssim ((1-c)R)^{-1}$.
\end{proof}

Remark that the constant $C_{\mathrm{int}}$ of \eqref{eq:CaccType2D} depends on $s$,
due to the usage of \eqref{eq:L3.8-KarMel19} in the proof above.

The Caccioppoli inequality \eqref{eq:CaccType2D} in Lemma~\ref{lem:CaccType2D} 
can be iterated on concentric balls to provide control of higher order derivatives 
by lower order derivatives locally, in the interior of the domain.
\begin{corollary}[High order interior Caccioppoli inequality]
\label{cor:CaccHighInt} 
Let $B_R \coloneqq  B_R(x_0) \subset \Omega \subset \R^d$ be an open ball of radius $R>0$ centered at $x_0 \in \Omega$, 
and let $B_{cR}$ be the concentric scaled ball of radius $cR$ with $c\in(0,1)$. 
Let $U$ satisfy \eqref{eq:extension2D}
for given data $f$ and $F$ with $\supp F \subset \R^d \times [0,H]$.

Then, 
there exists a constant $\gamma > 0$ (depending only on $\alpha$, $\Omega$, and $c$) 
such that for all 
$\beta \in \mathbb{N}_0^d$ with $p = \betam$ holds
\begin{multline}
\norm{\dbeta\nabla U}^2_{L^2_\alpha(B_{cR}^+)} \leq
 (\gamma p)^{2p} R^{-2p}  \norm{\nabla U}^2_{L^2_\alpha(B_R^+)} \\
  + \sum_{j=1}^p (\gamma p)^{2(p-j)} R^{2(j-p)}\left(\max_{\etam=j}\norm{\deta f}^2_{L^2(B_R)} 
  +\max_{\etam=j-1} \norm{\deta F}^2_{L^2_{-\alpha}(B_{R}^+)}\right).
\end{multline}
\end{corollary}
\begin{proof}
We start by noting that the case $p = 0$ is trivially true since empty sums are zero and $0^0 = 1$. 
For $p \ge 1$, we fix a multi index $\beta$ such that $\betam = p$. 
  As the $x$-derivatives commute with the differential operator in
  \eqref{eq:extension2D}, we have that $\dbeta U$ solves 
equation \eqref{eq:extension2D} with data $\dbeta F$ and $\dbeta f$. 
For given $c>0$, let
\begin{equation*}
  c_i = c + (i-1)\frac{1-c}{p}, \qquad i=1, \dots, p+1.
\end{equation*}
Then, we have $c_{i+1}R-c_iR = \frac{(1-c)R}{p}$ and $c_1 R = cR$ as well as $c_{p+1} R = R$.
For ease of notation and without loss of generality, we assume that $\beta_1>0$. 
Applying Lemma~\ref{lem:CaccType2D} iteratively on the sets $B_{c_iR}^+$ for $i>1$ provides 
\begin{align*}
 & \norm{\dbeta \nabla U}^2_{L^2_\alpha(B_{c R}^+)} \\
& \quad \leq \CacInt^{2} 
           \left(\frac{p^2}{(1-c)^2}R^{-2} \norm{\partial_x^{(\beta_1-1, \beta_2)}\nabla U}^2_{L^2_\alpha(B_{c_2 R}^+)}
 + C_{\rm loc}^2 \norm{\dbeta  f}^2_{L^2(B_{c_2 R})} + \norm{\partial_x^{(\beta_1-1, \beta_2)}F}^2_{L^2_{-\alpha}(B_{c_2 R}^+)}  \right) \\
  & \quad \leq \left(\frac{\CacInt p }{(1-c)}\right)^{2p} R^{-2p} \norm{\nabla U}^2_{L^2_\alpha(B_{R}^+)}
    + C_{\rm loc}^2 \sum_{j=1}^p\left(\frac{\CacInt p }{(1-c)}\right)^{2p-2j}R^{-2p+2j}\max_{\etam=j}\norm{\deta  f}^2_{L^2(B_{c_{p-j+2} R})} \\
  &\qquad + \sum_{j=0}^{p-1}\left(\frac{\CacInt p }{(1-c)}\right)^{2p-2j-2}R^{-2p+2j+2}\max_{\etam=j}\norm{\deta F}^2_{L^2_{-\alpha}(B_{c_{p-j+1} R}^+)}. 
\end{align*}
Choosing $\gamma = \max(C_{\rm loc}^2, 1)\CacInt/(1-c)$ concludes the proof.
\end{proof}

The previous Caccioppoli inequalities can also be localized in $y$. To that end, we recall the notation $\omega^{\theta} = \omega \times (0,\theta)$. 

\begin{corollary}\label{cor:localYCaccioppoli}
Let $B_R \coloneqq  B_R(x_0) \subset \Omega \subset \R^d$ be an open ball of radius $R>0$ centered at $x_0 \in \Omega$, 
and let $B_{cR}$ be the concentric scaled ball of radius $cR$ with $c\in(0,1)$. 
Let $0<\theta < \theta^\prime$.
Let $U$ satisfy \eqref{eq:extension2D}
for given data $f$ and $F$ with $\supp F \subset \R^d \times [0,H]$.
Then, 
there exists a constant $\gamma > 0$ (depending only on $\alpha$, $\Omega$, $H$, $\theta$, $\theta^\prime$, and $c$) 
such that there holds for all 
$\beta \in \mathbb{N}_0^d$ with $p = \betam$ 
\begin{multline*}
\norm{\dbeta\nabla U}^2_{L^2_\alpha(B_{cR}^{\theta})} \leq
 (\gamma p)^{2p} R^{-2p}  \norm{\nabla U}^2_{L^2_\alpha(B_R^{\theta^\prime})} \\
  + \sum_{j=1}^p (\gamma p)^{2(p-j)} R^{2(j-p)}\left(\max_{\etam=j}\norm{\deta f}^2_{L^2(B_R)} 
  +\max_{\etam=j-1} \norm{\deta F}^2_{L^2_{-\alpha}(B_{R}^+)}\right).
\end{multline*}
\end{corollary}
\begin{proof}
The proof is very similar to the proof of Corollary~\ref{cor:CaccHighInt}, which iterates Lemma~\ref{lem:CaccType2D}.
In fact, an (also in $y$) localized version of Lemma~\ref{lem:CaccType2D} can be obtained by replacing the cut-off function 
$\zeta = \zeta(x)$ in   Lemma~\ref{lem:CaccType2D} by a cut-off function with product structure 
\begin{align*}
\zeta(x,y) = \zeta_x(x) \zeta_y(y), \qquad \zeta_x \in C_0^\infty(B_R), \quad \zeta_y \in C_0^\infty(-\theta^\prime,\theta^\prime).
\end{align*}
Here, $\zeta_x$ is the cut-off function as stated in Lemma~\ref{lem:CaccType2D} 
and $\zeta_y$ satisfies $\zeta_y \equiv 1$ on $(-\theta ,\theta)$ as well as 
$\|\partial_y^j \zeta_y\|_{L^\infty(-\theta^\prime,\theta^\prime)} \leq C_\zeta {(\theta^\prime-\theta )}^{-j}$ for $j\in\{0,1\}$ 
with a constant $C_{\zeta}$ independent of $R$, $\theta$, $\theta'$. Hence $\|\nabla \zeta\|_{L^\infty(\R^d \times \R_+)} \lesssim ((1-c) R)^{-1} + (\theta'-\theta)^{-1}$. 
Then, tracking the arguments in the proof Lemma~\ref{lem:CaccType2D} leads to a variant
of estimate (\ref{eq:CaccType2D}) in which $B_{cR}^{+}$ is replace by $B^\theta_{cR}$, the set 
$B^+_R$ is replaced by $B_{R}^{\theta^\prime}$, and the factor $((1-c)R)^{-2}$ is replaced by $((1-c)R)^{-2} + (\theta'-\theta)^{-2}$. 
The statement of the present corollary
is then obtained by an iteration argument similar to that in Cor.~\ref{cor:CaccHighInt} 
using the nested sets $B^{\theta_i}_{c_i R}$ where $\theta_i = \theta + (i-1)\frac{\theta^\prime-\theta}{p}$. As $R \leq \operatorname{diam} \Omega$, 
one has $(\theta_{i+1}-\theta_i)^{-2} + (c_{i+1} R - c_{i} R)^{-2} \leq C p^2 R^{-2}/(1-c)^2$ for a $C > 0$ depending only on $\Omega$, $\theta$, $\theta'$. 
\end{proof}

\section{Local tangential regularity for the extension problem in 2d} 
\label{sec:LocTgReg}

Lemma~\ref{lem:regularity2D} provides global regularity for the solution $U$ of (\ref{eq:extension2D}).
In this section, we derive a localized version of Lemma~\ref{lem:regularity2D} for tangential derivatives of $U$, where we 
solely consider the case $d=2$. 

Lemma~\ref{lem:CaccType2D} is formulated as an interior regularity estimate as 
the balls are assumed to satisfy $B_R(x_0) \subset \Omega$. Since $u = 0$ on $\Omega^c$ (i.e., $u$
satisfies ``homogeneous boundary conditions''), one obtains estimates near $\partial\Omega$ for derivatives in the direction of an edge. 

\begin{lemma}[Boundary Caccioppoli inequality]
\label{lem:CaccType2D-bdy} 
Let $\mathbf{e} \subset \partial\Omega$ be an edge of the polygon $\Omega$. Let
$B_R\coloneqq B_R(x_0)$ be an open ball with radius $R>0$ and center $x_0 \in \mathbf{e}$ such that 
$B_R(x_0) \cap \Omega$ is a half-ball, and let $B_{cR}$ be the concentric scaled ball of radius $cR$ with $c \in (0,1)$.
Let $\zeta \in C^\infty_0({B_R})$ be a cut-off function with 
$0 \leq \zeta \leq 1$ and $\zeta \equiv 1$ on ${B_{cR}}$ as well as $\|\nabla \zeta\|_{L^\infty(B_R)} \leq C_\zeta ((1-c)R)^{-1}$ 
for some $C_\zeta > 0$ independent of $c$, $R$. 
Let $U$ satisfy \eqref{eq:extension2D}
for given data $f$ and $F$ with $\supp F \subset \R^d \times [0,H]$.

Then, there exists a constant $C > 0$ (depending only on $\alpha$, $\Omega$, $C_\zeta$) such that
\begin{align}
 \norm{D_{x_\parallel}\nabla U}^{2}_{L^2_\alpha(B_{cR}^+)} \leq C \left( ((1-c)R)^{-2} \norm{\nabla U}^{2}_{L^2_\alpha(B_R^+)}
 + \norm{\zeta D_{x_\parallel} f}^{2}_{H^{-s}(\Omega)} + \norm{F}^{2}_{L^2_{-\alpha}(B_R^+)}\right). 
\end{align}
Furthermore, $\|\zeta D_{x_\parallel} f\|_{H^{-s}(\Omega)} \leq C_{\rm loc} \|D_{x_\parallel} f\|_{L^2(B_R \cap \Omega)}$ 
for some $C_{\rm loc} > 0$ independent of $R$ and $c$ (cf.\ Lemma~\ref{lemma:localization-fractional-norms}).
\end{lemma}
\begin{proof}
The proof is almost verbatim the same as that of Lemma~\ref{lem:CaccType2D}. The key observation is that $
V = D^{-\tau}_{x_\parallel} (\zeta^2 D^\tau_{x_\parallel} U)$ with the difference quotient 
$$
D^\tau_{x_\parallel} w(x)\coloneqq  \frac{w(x + \tau \mathbf{e}_\parallel) - w(x)}{\tau}
$$
is an admissible test function. 
\end{proof}

Iterating the boundary Caccioppoli equation provides an estimate for higher order tangential derivatives.

\begin{corollary}[High order boundary Caccioppoli inequality]
  \label{cor:CaccHighBound}
Let $\mathbf{e} \subset \partial\Omega$ be an edge of $\Omega$. 
Let $B_R\coloneqq B_R(x_0)$ be an open ball with radius $R>0$ and center $x_0 \in \mathbf{e}$ such that 
$B_R(x_0) \cap \Omega$ is a half-ball, and let $B_{cR}$ be the concentric scaled ball of radius $cR$ with $c \in (0,1)$.
Let $U$ satisfy \eqref{eq:extension2D}
for given data $f$ and $F$ with $\supp F \subset \R^d \times [0,H]$.

Then, there exists a constant $\gamma>0$ (depending only on $\alpha$, $\Omega$ and $c$, but independent of the choice of $R>0$) 
such that for every $p \in \N_{0}$ there holds
\begin{align}
  \label{eq:CaccHighBound}
\|D_{x_\parallel}^p\nabla U\|^2_{L^2_\alpha(B_{cR}^+)} & \leq 
  (\gamma p)^{2p} R^{-2p} \|\nabla U\|^2_{L^2_\alpha(B_R^+)}  \\
\nonumber 
  & \qquad \mbox{} + \sum_{j=1}^p (\gamma p)^{2(p-j)} R^{2(j-p)}\left(\|D^j_{x_\parallel} f\|^2_{L^2(B_R)} 
  + \|D^{j-1}_{x_\parallel} F\|^2_{L^2_{-\alpha}(B_{R}^+)}\right).
\end{align}
\end{corollary}
\begin{proof}
  The statement follows from Lemma \ref{lem:CaccType2D-bdy} in the same way as
  Corollary~\ref{cor:CaccHighInt} follows from Lemma~\ref{lem:CaccType2D}.
\end{proof}

The term $\norm{\nabla U}_{L^2_{\alpha}(B_R^+)}$ in (\ref{eq:CaccHighBound}) 
is actually small for $R\rightarrow 0$ in the presence of regularity of $U$, which was asserted in 
Lemma~\ref{lem:regularity2D}; this is quantified in the following lemma. 
\begin{lemma}\label{lem:estH12D}
Let $ S_{R} \coloneqq  \{x \in \Omega \;\colon \; r_{\partial\Omega}(x) < R\}$ be the tubular neighborhood of 
$\partial\Omega$ of width $R>0$.
Then, for $t \in [0,1/2)$, there exists $\Creg > 0$ depending only on $t$ and $\Omega$ such that the solution $U$ of 
(\ref{eq:minimization})
satisfies 
\begin{align}
\label{eq:lem:estH12D-5}
R^{-2t} \|\nabla U\|^2_{L^2_{\alpha}(S^+_R)} 
\leq 
\|r^{-t}_{\partial\Omega} \nabla U \|^2_{L^2_{\alpha}(\Omega^+)} 
\leq 
C_{\rm reg} C_t N^2(U,F,f)
\end{align}
with the constant $C_t>0$ from Lemma~\ref{lem:regularity2D} and $N^2(U,F,f)$ given by (\ref{eq:CUFf}). 
\end{lemma}
\begin{proof}
The first estimate in (\ref{eq:lem:estH12D-5}) is trivial. For the second bound, we start by noting that
the shift result Lemma~\ref{lem:regularity2D} gives the global regularity
\begin{align}
\label{eq:lem:estH12D-10}
\int_{\Rpos} y^\alpha \norm{\nabla U(\cdot, y)}_{H^t(\Omega)}^2 dy \leq C_t N^2(U,F,f). 
\end{align}
For $t \in [0,1/2)$ and any $v \in H^t(\Omega)$, we have by, e.g., \cite[Thm.~{1.4.4.3}]{Grisvard} the embedding result 
$\|r^{-t}_{\partial\Omega} v\|_{L^2(\Omega)} \leq C_{\mathrm{reg}}
\|v\|_{H^t(\Omega)}$. Applying this embedding to $\nabla U(\cdot,y)$,
multiplying by $y^\alpha$, and integrating in $y$ yields (\ref{eq:lem:estH12D-5}). 
\end{proof}

 The following lemma provides a shift theorem for localizations
of tangential derivatives of $U$.

\begin{lemma}[High order localized shift theorem]
\label{lem:localhighregularity2D}
Let $U$ be the solution of 
(\ref{eq:minimization}).  
Let $x_0 \in \mathbf{e}$ for an edge $\mathbf{e} \in {\mathcal E}$ of the polygon $\Omega$. Let $R\in (0,1/2]$, and assume that $B_R(x_0) \cap \Omega$ is a half-ball. 
Let $\eta_x \in C^\infty_0(B_R(x_0))$, $\eta_y \in C^\infty_0(-H,H)$ with $\eta_y \equiv 1$ on $(-H/2,H/2)$ and 
$\|\nabla^j \eta_x\|_{L^\infty(B_R(x_0))} \leq C_\eta R^{-j}$, $j \in \{0,1,2\}$ as well as 
$\|\partial_y^j \eta_y\|_{L^\infty(-H,H)} \leq C_\eta H^{-j}$, $j \in \{0,1,2\}$, with a constant $C_\eta > 0$ independent of $R$ and $H$. Let $\eta(x,y):= \eta_x(x) \eta_y(y)$.
Then,
for $t \in [0,1/2)$, there is $C > 0$ independent of $R$ and $x_0$
 such that, for each $p\in\N$, 
the function $\tUp\coloneqq \eta D^p_{x_\parallel} U$ satisfies 
 \begin{align}
\label{eq:lem:localhighregularity2D-5}
\int_{\Rpos} y^\alpha \norm{\nabla \tUp(\cdot,y)}_{H^t(\Omega)}^2 dy 
\leq C  R^{-2p-1+2t}(\gamma p)^{2p}(1+\gamma p) \tNp(F, f),
\end{align}
where $\gamma$ is the constant in Corollary~\ref{cor:CaccHighBound} and
\begin{align}
\label{eq:Ntilde}
\tNp(F,f)& \coloneqq 
 \norm{f}^2_{H^1(\Omega)} + \norm{F}^2_{L^2_{-\alpha}(\R^2\times(0,H))} \\
\nonumber 
& \qquad \mbox{} + \sum_{j=2}^{p+1} (\gamma p)^{-2j}\bigg(2^j\max_{\betam =j}\|\dbeta f\|^2_{L^2(\Omega)}
  + 2^{j-1}\max_{\betam=j-1}\|\dbeta F\|^2_{L^2_{-\alpha}(\R^2\times (0,H))}\bigg).
\end{align}
In addition,
\begin{equation}
\label{eq:lem:localhighregularity2D-10}
\int_{\Rpos} y^\alpha \|r^{-t}_{\partial\Omega} \nabla \widetilde U^{(p)}(\cdot,y)\|^2_{L^2(\Omega)}\,dy 
\leq C R^{-2p-1+2t} (\gamma p)^{2p}(1+\gamma p) \widetilde N^{(p)}(F,f).
\end{equation}%
\end{lemma}
\begin{proof}
We abbreviate $\Uparp\coloneqq  \Dpar^p U$, $\tUp(x,y)\coloneqq  \eta(x,y) \Dpar^p U(x,y)$, 
$\Fparp = \Dpar^p F$, and $\fparp  = \Dpar^p f$. 
Throughout the proof we will use the fact that, 
for all $j\in \N$ and all sufficiently smooth functions $v$, 
we have 
\begin{equation*}
  |\Dpar^j v | \leq 2^{j/2}\max_{\betam=j} |\dbeta v|.
\end{equation*}
We also note that the assumptions on $\eta(x,y) = \eta_x(x) \eta_y(y)$ imply the existence of $\tilde C_\eta > 0$ (which absorbes the dependence
on $H$ that we do not further track) such that 
\begin{equation}
\|\nabla^j_{x} \partial_y^{j'} \eta\|_{L^\infty(\R^2 \times \R)} \leq \tilde C_\eta R^{-j}, \qquad j \in \{0,1,2\}, j' \in \{0,1,2\}. 
\end{equation}
{\bf Step 1.} (Localization of the equation). 
Using that $U$ solves the extension problem (\ref{eq:extension2D}), we obtain that the function 
$\tUp = \eta \Uparp$ satisfies the equation
\begin{align*}
  &\operatorname*{div} (y^\alpha \nabla \tUp) =   y^\alpha \operatorname{div}_x (\nabla_x \tUp) + \partial_y(y^\alpha \partial_y\tUp) 
   \\&\qquad\qquad =  y^\alpha\left((\Delta_x \eta) \Uparp + 2 \nabla_x \eta \cdot \nabla_x \Uparp +\eta \Delta_x \Uparp\right) + \eta\partial_y(y^\alpha \partial_y\Uparp) 
+ \partial_y (y^\alpha \Uparp \partial_y \eta ) + y^\alpha \partial_y \Uparp \partial_y \eta 
   \\&\qquad\qquad = y^\alpha\left((\Delta_x \eta) \Uparp + 2 \nabla_x \eta \cdot \nabla_x \Uparp\right) 
+ \partial_y (y^\alpha \Uparp \partial_y \eta ) + y^\alpha \partial_y \Uparp \partial_y \eta 
+ \eta   \operatorname*{div} (y^\alpha \nabla \Uparp) 
   \\& \qquad\qquad = y^\alpha\left((\Delta_x \eta) \Uparp + 2 \nabla_x \eta \cdot \nabla_x \Uparp\right)
+ \partial_y (y^\alpha \Uparp \partial_y \eta ) + y^\alpha \partial_y \Uparp \partial_y \eta
+ \eta \Fparp \eqqcolon  \tFp
\end{align*}
as well as the boundary conditions 
\begin{align*}
\partial_{n_\alpha} \tUp(\cdot,0) & = \eta(\cdot,0) \Dpar^p f  \eqqcolon  \tfp   &&\mbox{ on $\Omega$}, \\
\operatorname{tr}\tUp & = 0  &&\mbox{ on $\Omega^c$.}
\end{align*}
By the support properties of the cut-off function $\eta$, we have $\operatorname{supp} \tFp \subset \overline{B_R}(x_0) \times [0,H] \subset \R^2 \times [0,H]$.
By Lemma~\ref{lem:regularity2D}, 
for all $t \in [0,1/2)$, there is a $C_t>0$ such that
\begin{equation}
  \label{eq:reg2D_in_highord}
  \int_{\Rpos}y^\alpha \|\nabla\tUp(\cdot, y)\|^2_{H^t(B_{\widetilde{R}})}dy \leq C_t N^2(\tUp, \tFp, \tfp),
\end{equation}
where $B_{\widetilde{R}}$ is a ball containing $\overline{\Omega}$.
By \eqref{eq:CUFf}, we have to estimate $N^2(\tUp, \tFp, \tfp)$, i.e., 
$\|\nabla \tUp\|_{L^2_{\alpha}(\R^2\times \Rpos)}$, $\|\tFp\|_{L^2_{-\alpha}(\R^2 \times (0,H))}$,  and 
$\|\tfp\|_{H^{1-s}(\Omega)}$. 
Let $\gamma$ be the constant introduced in Corollary~\ref{cor:CaccHighBound}.
We note that by (\ref{eq:CUFf-simplified}) there exists $C_N>0$ such that, for
all $p\in\N_{0}$,
\begin{equation}
\label{eq:Ntilde-10}
N^2(U,F,f)  \leq C_N \tNp(F,f).
\end{equation}
{\bf Step 2.} (Estimate of $\|\nabla \tUp\|_{L^2_{\alpha}(\R^2\times \Rpos)}$).
We write
\begin{align}\label{eq:lem:localregularity2D-00}
\nonumber 
\|\nabla \tUp\|^2_{L^2_{\alpha}(\R^2\times \Rpos)} &\leq 
2\|\nabla \eta\|^2_{L^\infty}  \|\nabla_x \Uparpmone\|^2_{L^2_{\alpha}(B^+_R)} + 2\|\eta\|^2_{L^\infty} \|\nabla \Uparp\|^2_{L^2_{\alpha}(B^+_{R})} 
  \\ 
&\leq 2\tilde C_{\eta}^2 \left( R^{-2}
\|\nabla\Uparpmone\|^2_{L^2_{\alpha}(B^+_{R})} + \|\nabla \Uparp\|^2_{L^2_{\alpha}(B^+_{R})}\right). 
\end{align}
We employ Corollary~\ref{cor:CaccHighBound} with a ball $B_{2R}$ and $c=1/2$ as well as Lemma~\ref{lem:estH12D} to obtain for $p \in \N_0$
\begin{align}\label{eq:lem:localregularity2D-10}
\nonumber 
 \|\nabla \Uparp\|^2_{L^2_{\alpha}(B^+_{R})} &\leq 
         (2R)^{-2p}(\gamma p)^{2p}\bigg(\|\nabla U\|^2_{L^2_{\alpha}(B^+_{2R})} + \sum_{j=1}^p (2R)^{2j} (\gamma p)^{-2j} \Big(\|D^j_{x_\parallel} f\|^2_{L^2(B_{2R})}
  + \|D^{j-1}_{x_\parallel} F\|^2_{L^2_{-\alpha}(B_{2R}^+)}\Big)\bigg) \\
  & \leq 
          (2R)^{-2p}(\gamma p)^{2p} \bigg( \|\nabla U\|^2_{L^2_{\alpha}(B^+_{2R})} \nonumber \\
          & \quad + (2R)^{2}\sum_{j=1}^p (2R)^{2(j-1)} (\gamma p)^{-2j}\Big(2^j\max_{\betam=j} \|\dbeta f\|^2_{L^2(B_{2R})}
  + 2^{j-1}\max_{\betam=j-1} \|\dbeta F\|^2_{L^2_{-\alpha}(B_{2R}^+)} \Big)\bigg) \nonumber \\
\overset{R\leq 1/2, \text{L.\ref{lem:estH12D}}}&{\leq}
                                      (2R)^{-2p}(\gamma p)^{2p} \Big(  \left(\Creg C_t R^{2t}+(2R)^2 2 \gamma^{-2} \right)N^2(U,F,f)  + (2R)^2\tNp(F, f) \Big) \nonumber
  \\  \overset{t <1/2, (\ref{eq:Ntilde-10})}&{\leq}   (2R)^{-2p}(\gamma p)^{2p}
\underbrace{ (\Creg C_t (1+8 \gamma^{-2}) C_N + 4)}_{=: \CregN} R^{2t} \tNp(F,f). 
\end{align} 
For $p \in \N$, we 
apply (\ref{eq:lem:localregularity2D-10}) to the $(p-1)^{th}$ derivative and exploit the structure of the expression 
$(\gamma (p-1))^{2p-2} \tNpmone(F,f)$ to get 
\begin{align}\label{eq:lem:localregularity2D-20}
\nonumber 
\|\nabla \Uparpmone\|^2_{L^2_{\alpha}(B^+_{R})} &\leq 
 (2R)^{-2(p-1)} \CregN R^{2t}
(\gamma (p-1))^{2(p-1)} 
\tNpmone(F,f) \\
& \leq  (2R)^{-2(p-1)}
\CregN R^{2t}  \max\{1,\gamma^{-2}\} 
(\gamma p)^{2p}
\tNp(F,f). 
\end{align}
Inserting \eqref{eq:lem:localregularity2D-10} and \eqref{eq:lem:localregularity2D-20} into \eqref{eq:lem:localregularity2D-00} provides the estimate
\begin{align*}
\|\nabla \tUp\|^2_{L^2_{\alpha}(\R^2\times \Rpos)} \leq C  R^{-2p+2t}(\gamma p)^{2p} \tNp(F,f)
\end{align*}
with a constant $C>0$ depending only on the constants $\Creg$, $C_t$, $\tilde C_{\eta}$, $C_N$, and $\gamma$.

{\bf Step 3.} (Estimate of $\|\tFp\|_{L^2_{-\alpha}(\R^2 \times \Rpos)}$).
We treat the five terms appearing in $\|\tFp\|_{L^2_{-\alpha}(\R^2 \times \Rpos)}$ separately. 
With \eqref{eq:lem:localregularity2D-10}, we obtain
\begin{align*}
\norm{ y^\alpha\nabla_x \eta \cdot \nabla_x \Uparp}^2_{L^2_{-\alpha}(\R^2 \times (0,H))} &= 
\norm{\nabla_x \eta \cdot \nabla_x \Uparp}^2_{L^2_{\alpha}(\R^2 \times \Rpos)} \leq 
C_\eta^2 \frac{1}{R^2} \norm{\nabla_x \Uparp}^2_{L^2_{\alpha}(B_{R}^+)} 
\\ 
\overset{ (\ref{eq:lem:localregularity2D-10})}&{\leq}
(2R)^{-2p}(\gamma p)^{2p}C_\eta^2
\CregN
R^{-2+2t} \tNp(F,f).
\end{align*}
Similarly, we get
\begin{align*}
\norm{y^\alpha(\Delta_x \eta) \Uparp}^2_{L^2_{-\alpha}(\R^2 \times (0,H))}
&=  \norm{(\Delta_x \eta) \Uparp}^2_{L^2_{\alpha}(B^+_R)} \leq
C_\eta^2 \frac{1}{R^4} \norm{\nabla \Uparpmone}^2_{L^2_{\alpha}(B_{R}^+)}
\\ 
  \overset{\text{\eqref{eq:lem:localregularity2D-20}}}&{\leq} 4 (2R)^{-2p}(\gamma p)^{2p} C_\eta^2 
\CregN
R^{-2+2t} \tNp(F,f).
\end{align*}
Next, we estimate 
\begin{align*}
  \|\eta \Fparp\|^2_{L^2_{-\alpha} (\R^2 \times (0,H))}
  \leq \|\Fparp\|^2_{L^2_{-\alpha} (B_R^+)} \leq 2^p \max_{\betam=p} \|\dbeta F\|_{L^2_{-\alpha} (B_R^+)}^2 \leq (\gamma p)^{2p+2} \tNp(F, f). 
\end{align*}
Finally, for the term 
$\partial_y (y^\alpha \Uparp \partial_y \eta ) + y^\alpha \partial_y \Uparp \partial_y \eta$, we observe that $\partial_y \eta$ vanishes near $y = 0$ so that the weight
$y^\alpha$ does not come into play as it can be bounded from above and below by positive constants depending only on $H$. We arrive at 
\begin{align*}
\norm{\partial_y (y^\alpha \Uparp \partial_y \eta ) + y^\alpha \partial_y \Uparp \partial_y \eta}_{L^2_{-\alpha}(\R^2 \times (0,H))} 
&\leq 
C \left( H^{-2} \|\Uparp\|_{L^2_\alpha(B_R \times (0,H))} + H^{-1}\|\nabla \Uparp\|_{L^2_\alpha(B^+_R)} \right) \\
&\stackrel{\eqref{eq:lem:localregularity2D-10},\eqref{eq:lem:localregularity2D-20}}{\leq } C_H (\gamma p)^{2p}R^{-2p+2t}\tNp(F,f),
\end{align*}
for suitable $C_H > 0$ depending on $H$. 

{\bf Step 4.} (Estimate of $\|\tfp\|_{H^{1-s}(\Omega)}$.)
Here, we use
Lemma~\ref{lemma:localization-fractional-norms} and $R<1/2$ together with $s<1$ to obtain 
\begin{align*}
\|\tfp\|_{H^{1-s}(\Omega)}^2 & 
\leq 2\Cloctwo^2 C_\eta^2 \left( 9R^{2s-2} \|\Dpar^p f\|^2_{L^2(\Omega)} + |\Dpar^p f|^2_{H^{1-s}(\Omega)} \right) 
  \\ &
  \leq C\Cloctwo^2C_\eta^2 R^{2s-2}\left(2^p \max_{\betam=p}\| \dbeta f\|^2_{L^2(\Omega)}  + 2^{p+1} \max_{\betam=p+1}\|\dbeta f\|^2_{L^2(\Omega)}  \right)
  \\ &
  \leq C \Cloctwo^2C_\eta^2 R^{2s-2}(\gamma p)^{2p}(1+(\gamma p)^2) \tNp(F, f) 
\end{align*}
with a constant $C>0$ depending only on $\Omega$ and $s$.

{\bf Step 5.} (Putting everything together.)
Combining the above estimates, we obtain that there exists a constant $C>0$ depending
only on $\Creg$, $C_t$, $\tilde C_\eta$, $C_N$, $\Cloctwo$, $H$, $\gamma$, $\Omega$, $s$ such that
\begin{align*}
  & N^2(\tUp,\tFp,\tfp)
  \\ & \qquad = 
\left(
\|\nabla \tUp\|^2_{L^2_\alpha(\R^2 \times \Rpos)} + 
\|\nabla \tUp\|_{L^2_\alpha(\R^2 \times \Rpos)} \|\tFp\|_{L^2_{-\alpha}(\R^2\times (0,H))} + 
       \|\nabla \tUp\|_{L^2_\alpha(\R^2 \times \Rpos)} \|\tfp\|_{H^{1-s}(\Omega)}   \right)
  \\ & \qquad
       \leq C \left[ R^{-2p+2t}(\gamma p)^{2p} + R^{-p+t}(\gamma p)^p R^{-p-1+t}(\gamma p)^{p}(1+\gamma p)  + R^{-p+t}(\gamma p)^p R^{s-1}(\gamma p)^{p}(1+\gamma p) \right]\tNp(F,f)
  \\ 
       \overset{R \leq 1, t<1/2}&{\qquad \leq} C R^{-2p-1+2t}(\gamma p)^{2p}(1+\gamma p)\tNp(F,f).
\end{align*}
Inserting this estimate in \eqref{eq:reg2D_in_highord} concludes the proof of (\ref{eq:lem:localhighregularity2D-5}). 

{\bf Step 6:} The estimate (\ref{eq:lem:localhighregularity2D-10}) 
follows from \cite[Thm.~{1.4.4.3}]{Grisvard},
which gives
\begin{equation*}
  \int_{\Rpos} y^\alpha \|r^{-t}_{\partial\Omega} \nabla \widetilde U^{(p)}(\cdot,y)\|^2_{L^2(\Omega)}\,dy 
\leq C 
\int_{\Rpos} y^\alpha \|\nabla \widetilde U^{(p)}(\cdot,y)\|^2_{H^t(\Omega)}\,dy,
\end{equation*}
and from \eqref{eq:lem:localhighregularity2D-5}. 
\end{proof}

\section{Weighted $H^p$-estimates in polygons}
\label{sec:WghHpPolygon}
In this section, 
we derive higher order weighted regularity results, 
at first for the extension problem and finally for the fractional PDE.
This is our main result, Theorem~\ref{thm:mainresult}.
%
\subsection{Coverings}
A main ingredient in our analysis are suitable localizations of \emph{vertex neighborhoods} $\omega_{\mathbf{v}}$ and 
\emph{vertex-edge neighborhoods} $\omega_{\mathbf{ve}}$ near a vertex $\mathbf{v}$ and of \emph{edge neighborhoods} $\omega_{\mathbf{e}}$ 
near an edge $\mathbf{e}$. This is achieved by covering such neighborhoods by balls or half-balls with the following two properties: 
a) their diameter is proportional to the distance to vertices or edges and b) scaled versions of these balls/half-balls satisfy a locally finite overlap property.

We start by recalling a lemma that follows from Besicovitch's Covering Theorem:
\begin{lemma}[{\cite[Lem.~{A.1}]{melenk-wohlmuth12}, \cite[Lem.~{A.1}]{horger-melenk-wohlmuth13}}]
\label{lemma:MW}
Let $\omega\subset {\mathbb R}^d$ be bounded, open and $M \subset \partial\omega$ be closed. Fix $c$, $\zeta \in (0,1)$ such that 
$1 -c (1+\zeta) \eqqcolon c_0 > 0$. For each $x \in \omega$, let $B_x\coloneqq  \overline{B}_{c \operatorname{dist}(x,M)}(x)$ be the closed ball of 
radius $c \operatorname{dist}(x,M)$ centered at $x$, and let
$\widehat{B}_x\coloneqq  \overline{B}_{(1+\zeta) c \operatorname{dist}(x,M)}(x)$ be the 
stretched closed ball of radius $(1+\zeta) c \operatorname{dist}(x,M)$ centered at $x$.  Then, there is a countable set 
$(x_i)_{i \in {\mathcal I}} \subset  \omega$ (for some suitable index set ${\mathcal I} \subset {\mathbb N}$) and a number $N \in {\mathbb N}$ 
depending solely on $d$, $c$, $\zeta$ with the following properties:
\begin{enumerate}
\item 
(covering property) $\bigcup_{i} B_{x_i} \supset \omega$. 
\item
(finite overlap) for $x \in {\mathbb R}^d$ it holds that $\operatorname{card}\{i\,|\, x \in \widehat{B}_{x_i} \} \leq N$. 
\end{enumerate}
\end{lemma}
\begin{proof}
The lemma is taken from \cite[Lem.~{A.1}]{melenk-wohlmuth12} except that there 
$x \in \omega$ in the condition of finite overlap is assumed. Inspection of the proof shows that this condition can be relaxed as given here. 
Note that the proof of \cite[Lem.~{A.1}]{melenk-wohlmuth12} required the balls $B_{x_i}$ to be non-degenerate, 
which is ensured in the present setting of $M \subset \partial\omega$.
\end{proof}
In the next lemma, we introduce a covering of $\omegac$, see Figure \ref{fig:covering-vertex}.

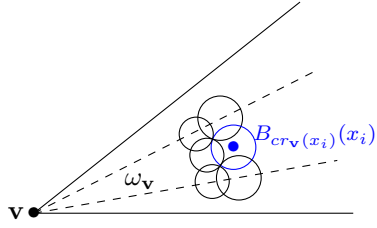
\begin{figure}
\begin{center} 
\begin{tikzpicture}[scale=0.7]
\draw (0,0) node {\textbullet} node[left]{$\mathbf{v}$}; 
  \draw[-] (0, -0) -- (6, 0);
  \draw[-] (0, -0) -- (5, 4);
  \draw[dashed] (0,-0) -- (5.7,1); 
  \draw[dashed] (0,-0) -- (5.3,2.7); 
\draw (2,0.6) node {$\omega_{\mathbf{v}}$};
\newcommand{\cirs}{0.42}
\draw[blue] (3.75,1.25) node {\textbullet};
\draw[blue] (3.95,1.45) node[right] {\footnotesize $B_{cr_\mathbf{v}(x_i)}(x_i)$}; 
\draw (3.5,3*3/5) circle (\cirs);
\draw[blue] (3.75,1.25) circle (\cirs);
\draw (3.85,3.35/5) circle (\cirs);
\newcommand{\cirss}{0.32}
\draw (3.05,1.5) circle (\cirss);
\draw (3.25,1.1) circle (\cirss);
\draw (3.35,3.0/5) circle (\cirss);
\end{tikzpicture}
\end{center}
\caption{\label{fig:covering-vertex} Covering of ``vertex cones'' such as
  $\omegac$ by union of balls $B_{cr_\mathbf{v}(x_i)}(x_i)$ with fixed $c \in (0,1)$. }
\end{figure}

\begin{lemma}[covering of $\omega_{\mathbf v}$]
\label{lemma:covering-omega_c}
Given ${\mathbf v} \in {\mathcal V}$ and $\omegaeps>0$, there are $0 < c < \widehat c < 1$ and points $(x_i)_{i \in {\mathbb N}}\subset \omega_{\mathbf v} = \omega^\omegaeps_{\mathbf v}$ such that 
the collections 
${\mathcal B}\coloneqq  \{B_i\coloneqq  B_{c \operatorname{dist}(x_i,{\mathbf v})}(x_i) \,|\, i \in {\mathbb N}\}$ 
and 
$\widehat {\mathcal B}\coloneqq  \{\widehat B_i\coloneqq  B_{\widehat c \operatorname{dist}(x_i,{\mathbf v})}(x_i) \,|\, i \in {\mathbb N}\}$ 
of (open) balls satisfy the following conditions: 
the balls from ${\mathcal B}$ cover $\omega_{\mathbf v}$; the balls from $\widehat{\mathcal B}$ satisfy
a finite overlap property with overlap constant $N$ depending only on the spatial dimension $d = 2$ and $c$, $\widehat c$; 
the balls from $\widehat{\mathcal B}$ are contained in $\Omega$.  
Furthermore, for every $\delta> 0$, there is $C_\delta> 0$ (depending additionally on
$\delta$) such that with the radii $R_i\coloneqq  \widehat c \operatorname{dist}(x_i,{\mathbf v})$ it holds that
\begin{equation}
\label{eq:lemma:covering-omega_c-10}
\sum_{i} R_i^\delta \leq C_\delta. 
\end{equation}
\end{lemma}
\begin{proof}
Apply Lemma~\ref{lemma:MW} with $M = \{{\mathbf v}\}$ and sufficiently small parameters $c$, $\zeta>0$. 
Note that by possibly slightly increasing the parameter 
$c$, one can ensure that the open balls rather than the closed balls given by Lemma~\ref{lemma:MW} cover $\omega_{\mathbf v}$. 
Also, since $c < 1$, 
the index set ${\mathcal I}$ of Lemma~\ref{lemma:MW} cannot be finite so that ${\mathcal I} = {\mathbb N}$.  

To see (\ref{eq:lemma:covering-omega_c-10}), we compute with the spatial dimension $d = 2$ 
\begin{align*}
\sum_i R_i^\delta = 
\sum_i R_i^{\delta -d} R_i^d \lesssim \sum_i \int_{\widehat B_i} r_{\mathbf v}^{\delta-d}\,dx 
\stackrel{\text{finite overlap}}{\lesssim} \int_{\Omega} r_{\mathbf v}^{\delta-d}\,dx < \infty. 
\end{align*}
\end{proof}
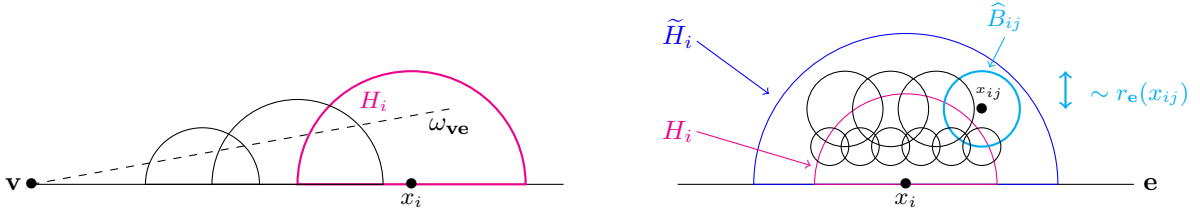
\begin{figure}
  \begin{center}
\begin{tikzpicture}
\draw (0,0) node {\textbullet} node[left]{$\mathbf{v}$}; 
  \draw[-] (0, -0) -- (7, 0);
  \draw[dashed] (0, -0) -- (5.5, 1) node[below] {$\omega_{\mathbf{ve}}$};
\newcommand{\cirs}{1.5}
\draw[magenta,thick] (5-\cirs,0) -- (5+\cirs, 0) arc (0:180:\cirs) --cycle; 
\draw (5,0) node {\textbullet} node[below]{\small $x_i$}; 
\draw[magenta] (4.5,1.1) node {\small $H_i$ }; 
\draw (3.5-0.75*\cirs,0) -- (3.5+0.75*\cirs, 0) arc (0:180:0.75*\cirs) --cycle; 
\draw (2.25-0.5*\cirs,0) -- (2.25+0.5*\cirs, 0) arc (0:180:0.5*\cirs) --cycle; 
\newcommand{\offs}{6.5}
\renewcommand{\cirs}{2.0}
  \draw[-] (\offs+2, 0) -- (\offs+8, 0) node[right] {$\mathbf{e}$};
\draw [blue] (\offs+5-\cirs,0) -- (\offs+5+\cirs,0) arc (0:180:\cirs) --cycle; 
\draw [magenta] (\offs+5-0.6*\cirs,0) -- (\offs+5+0.6*\cirs,0) arc (0:180:0.6*\cirs) --cycle; 
\draw (\offs+5,0) node {\textbullet} node[below]{\small $x_i$}; 
\newcommand{\cirss}{0.25*\cirs}
\draw[thick,cyan] (\offs+6,1) circle (\cirss); 
\draw (\offs+6,1) node {\textbullet}; 
\draw (\offs+6.1,1.0) node[above]{\tiny $x_{ij}$} ; 
\draw (\offs+5.4,1) circle (\cirss); 
\draw[cyan] (\offs+6.3,1.9) node[above]{\small $\widehat{B}_{ij}$};
\draw[cyan,thin,<-] (\offs+6.1,1.55) -- (\offs+6.3,1.95) ; 
 
    \draw[cyan,thick,<->] (\offs+7.1, 1) -- (\offs+7.1, 1+\cirss);
    \draw[cyan] (\offs+7.3,1.2) node[right] {\small $ \sim r_{\mathbf{e}}(x_{ij})$};
\draw (\offs+4.8,1) circle (\cirss); 
\draw (\offs+4.2,1) circle (\cirss); 
\draw (\offs+6,0.5) circle (0.5*\cirss); 
\draw (\offs+5.6,0.5) circle (0.5*\cirss); 
\draw (\offs+5.2,0.5) circle (0.5*\cirss); 
\draw (\offs+4.8,0.5) circle (0.5*\cirss); 
\draw (\offs+4.4,0.5) circle (0.5*\cirss); 
\draw (\offs+4.0,0.5) circle (0.5*\cirss); 
%
%
\draw[blue,thin,->] (\offs+2.25,1.9) -- (\offs+3.2,1.2) ; 
\draw (\offs+2,2) node [blue] {$\widetilde{H}_i$}; 
\draw[magenta,thin,->] (\offs+2.25,0.7) -- (\offs+3.75,0.25) ; 
\draw (\offs+2,0.7) node [magenta] {$H_i$}; 
\end{tikzpicture}
  \end{center}
\caption{\label{fig:covering-edge-vertex} Covering of $\omega_{\mathbf{v}\mathbf{e}}$. 
Left: the half-balls $H_i$ constructed in Lemma~\ref{lemma:covering-omega_ce}. 
Right: covering of $H_i$ by balls $B_{ij}$ such that the larger balls 
$\widehat{B}_{ij}$ are contained in a ball $\widetilde{H}_i$. 
For better illustration, only the larger balls  
$\widehat{B}_{ij}$ are shown, the balls $B_{ij}$ 
are included therein and still provide a covering of $H_i$.
        }
\end{figure}
We now introduce a covering of vertex-edge neighborhoods $\omegace$. 
We start by a
covering of half-balls resting on the edge $\mathbf{e}$ and with 
size proportional to the distance from the vertex, see Figure~\ref{fig:covering-edge-vertex} (left).
\begin{lemma}[covering of $\omega_{{\mathbf v}{\mathbf e}}$]
\label{lemma:covering-omega_ce}
Given ${\mathbf v} \in {\mathcal V}$, ${\mathbf e} \in {\mathcal E}({\mathbf v})$, there is $\omegaeps> 0$ and parameters $0 < c< \widehat c < 1$ 
as well as points $(x_i)_{i \in {\mathbb N}} \subset {\mathbf e}$ such that the following holds: 
\begin{enumerate}[(i)]
\item 
the sets $H_i\coloneqq  B_{c \operatorname{dist}(x_i,{\mathbf v})}(x_i) \cap \Omega$ are half-balls and the collection 
${\mathcal B}\coloneqq  \{H_i\,|\, i \in {\mathbb N}\}$ covers $\omega_{{\mathbf v}{\mathbf e}} = \omega^\omegaeps_{{\mathbf v}{\mathbf e}}$.  

\item
The collection $\widehat{\mathcal B}\coloneqq  \{\widehat H_i\coloneqq  B_{\widehat c \operatorname{dist}(x_i,{\mathbf v})}(x_i) \cap \Omega\}$ is a collection of half-balls and 
satisfies a finite overlap property, i.e., there is $N > 0$ depending only on the spatial dimension $d = 2$ and the parameters $c$, $\widehat c$ such that 
for all $x \in \R^2$ it holds that $\operatorname{card} \{i\,|\, x \in \widehat H_i\} \leq N$. 
\end{enumerate}
Furthermore, for every $\delta>0$ there is $C_\delta>0$ such that for the radii
$R_i\coloneqq  \widehat c \operatorname{dist}(x_i,{\mathbf v})(x_i)$ it holds that 
$\sum_i R^\delta_i \leq C_\delta$.  
\end{lemma}
\begin{proof}
Let $\widetilde{\mathbf e}$ be the (infinite) line containing ${\mathbf e}$. 
We apply Lemma~\ref{lemma:MW} to the 1D line segment ${\mathbf e}\cap B_{\omegaeps}({\mathbf v})$ (for some sufficiently small $\omegaeps$) 
 and $M\coloneqq \{\mathbf v\}$ and the parameter $c$ sufficiently small so that 
$B_{2 c \operatorname{dist}(x,{\mathbf v})}(x) \cap \Omega $ is a half-ball for all $x \in {\mathbf e} \cap B_{\omegaeps}({\mathbf v})$. 
Lemma~\ref{lemma:MW} provides a collection $(x_i)_{i \in {\mathbb N}} \subset {\mathbf e}$ such the balls 
$B_i\coloneqq  B_{c \operatorname{dist}(x_i,{\mathbf v})}(x_i) \subset {\mathbb R}^2$
and the stretched balls $\widehat B_i\coloneqq  B_{c (1+\zeta)\operatorname{dist}(x_i,{\mathbf v})}(x_i) \subset {\mathbb R}^2$ (for suitable, sufficiently small $\zeta$)
satisfy the following: the intervals $\{B_i\cap \widetilde{\mathbf e}\,|\, i \in {\mathbb N}\}$  
cover $B_{\omegaeps}({\mathbf v}) \cap \mathbf e$, and the intervals 
$\{\widehat B_i\cap \widetilde {\mathbf e}\,|\, i \in {\mathbb N}\}$ satisfy a finite overlap condition on $\widetilde{\mathbf e}$. 
By possibly slightly increasing the parameter $c$ (e.g., by replacing $c$ with $c(1+\zeta/2)$), 
the newly defined balls $B_i$ then cover a set 
$\omega^{\omegaeps}_{{\mathbf v}{\mathbf e}}$ 
for a possibly reduced $\omegaeps$. 
It remains to see that the balls $\widehat B_i$ satisfy a finite overlap condition on ${\mathbb R}^2$: 
given $x \in \widehat B_i$,
its projection $x_{\mathbf e}$ onto $\widetilde{\mathbf e}$ satisfies $x_{\mathbf e} \in \widehat B_i\cap \widetilde{\mathbf{e}}$ 
since $x_i \in {\mathbf e} \subset \widetilde{\mathbf e}$. 
This 
implies that the overlap constants of the balls $\widehat B_i$ in ${\mathbb R}^2$ is the same as the overlap constant of the intervals
$\widehat B_i \cap \widetilde{\mathbf e}$ in $\widetilde{\mathbf e}$. 
The
half-balls $H_i\coloneqq B_i \cap \Omega$ and $\widehat H_i\coloneqq  \widehat B_i \cap \Omega$ 
have the stated properties. 

Finally, the convergence of the sum $\sum_i R_i^\delta$ is shown 
by the same arguments as in Lemma~\ref{lemma:covering-omega_c}. 
\end{proof}
We will also need a covering of the half-balls $H_i$ constructed in
Lemma~\ref{lemma:covering-omega_ce}, which we introduce in the next lemma.
See also Figure~\ref{fig:covering-edge-vertex} (right).
\begin{lemma}
\label{lemma:subcovering-omega_ce}
Let ${\mathcal B} = \{H_i\,|\, i \in {\mathbb N}\}$ and 
$\widehat {\mathcal B} = \{\widehat H_i\,|\, i \in {\mathbb N}\}$ 
be constructed in Lemma~\ref{lemma:covering-omega_ce}.
Fix a $\widetilde c \in (c , \widehat c)$ with $c$, $\widehat c$ from Lemma~\ref{lemma:covering-omega_ce} and define the 
collection $\widetilde{\mathcal B}\coloneqq  \{\widetilde H_i\coloneqq  B_{\widetilde cr_{\mathbf{v}}(x_i)}(x_i)\cap\Omega\,|\, i \in {\mathbb N}\}$ 
of half-balls intermediate to the half-balls $H_i$ and $\widehat H_i$. 
 
There are constants $0 < c_1 < \widehat c_1 < 1$ such that the following holds: for each $i$, there are points $(x_{ij})_{j \in {\mathbb N}} \subset H_i$ 
such that 
the collection ${\mathcal B}_i\coloneqq  \{B_{ij}\coloneqq  B_{c_1r_{\mathbf{e}}(x_{ij})}(x_{ij})\}$ covers $H_i$ 
and the collection 
 $\widehat {\mathcal B}_i\coloneqq  \{\widehat B_{ij}\coloneqq  B_{\widehat c_1 r_\mathbf{e}(x_{ij})} (x_{ij})\}$ 
satisfies 
$\widehat B_{ij} \subset \widetilde H_i$ for all $j$ as well as a finite overlap
property, i.e., there is $N > 0$ independent of $i$ such that for all $x \in
\R^2$ it holds that 
$ \operatorname{card}\{j\,|\,  x \in \widehat B_{ij}\} \leq N$. 
\end{lemma}
\begin{proof} 
We apply Lemma~\ref{lemma:MW} with $M = \{\mathbf e\}$ and $\omega = H_i$. The
parameters $c$ and $\zeta$ are chosen small enough so that the balls $B_x$ in 
Lemma~\ref{lemma:MW} satisfy $\widehat B_x \subset \widetilde H_i$. 
Then, the lemma follows from Lemma~\ref{lemma:MW}. 
\end{proof}

\subsection{Weighted $H^p$-regularity for the extension problem}

To illustrate the techniques, we start with the simplest case of estimates 
in vertex neighborhoods $\omega_{\mathbf{v}}$. 
It is worth stressing that we have 
\begin{align*}
r_{\mathbf{e}} \sim r_{\mathbf{v}} \qquad \mbox{ on $\omega_{\mathbf{v}}$}. 
\end{align*}
The following lemma provides higher order regularity estimates in a vertex weighted norm for solutions to the Caffarelli-Silvestre extension problem with smooth data.

  \begin{lemma}[Weighted $H^p$-regularity in $\omega_{\mathbf{v}}$] 
\label{lemma:regularity-omega_c}
Let $\omega_{{\mathbf v}} = \omega^\omegaeps_{{\mathbf v}}$ be given for some $\omegaeps > 0$ and ${\mathbf v} \in {\mathcal V}$.  Let $U$ be the solution of 
(\ref{eq:minimization}).
There is $\gamma > 0$ depending only on $s$, $\Omega$, $\omega_{{\mathbf v}}$, and $H$, 
and for every $\varepsilon \in (0,1)$, there exists $C_\varepsilon>0$ 
depending additionally on $\varepsilon$
such that for all $\beta\in \N^2_0$ there holds with $p = \betam$
\begin{align*}
  \|r_{\mathbf{v}}^{p-1/2+\varepsilon} \dbeta \nabla U\|_{L^2_{\alpha}(\omega^+_{\mathbf{v}})}^2 &\leq
C_\varepsilon \gamma^{2p+1}p^{2p} \bigg[ \|f\|^2_{H^{1}(\Omega)} + \|F\|^2_{L^2_{-\alpha}(\R^2\times(0,H))}   
       \\ & \qquad +
       \sum_{j=2}^{p+1} p^{-2j}\left(\max_{\etam=j}\norm{\deta f}^2_{L^2(\Omega)} 
  +\max_{\etam=j-1} \norm{\deta F}^2_{L^2_{-\alpha}(\R^2\times(0,H))} \right)\bigg].
\end{align*}
\end{lemma}
\begin{proof} 
The case $p = 0$ follows from Lemma~\ref{lem:estH12D} and the estimates (\ref{eq:CUFf}), (\ref{eq:CUFf-simplified}). We therefore assume $p \in \N$.
 
Let the covering $\omega_{\mathbf{v}} \subset \bigcup_i B_i$ with $B_i = B_{c \operatorname{dist}(x_i,{\mathbf v})}(x_i)$ and stretched 
balls $\widehat B_i = B_{\widehat c \operatorname{dist}(x_i,{\mathbf v})}(x_i)$ be given by Lemma~\ref{lemma:covering-omega_c}. It will be convenient
to denote $R_i\coloneqq  \widehat c \operatorname{dist}(x_i,{\mathbf v})$ the radius of the ball $\widehat B_i$ and to note that, for some $C_B > 0$, 
\begin{equation}
\label{eq:estimate-r_c}
\forall i \in {\mathbb N}\quad \forall x \in \widehat B_i 
\qquad 
C^{-1}_B R_i \leq 
r_{\mathbf v}(x) \leq C_B  R_i. 
\end{equation}
We assume (for convenience) that $R_i \leq 1/2$ for all $i$. 

Let $\beta$ be a multi index and $p = \betam$. 
By (\ref{eq:Ntilde-10}) there is $C_N > 0$ independent of $p$ such that 
$N^2(U,F,f) \leq C_N \widetilde N^{(p)}(F,f)$,  
where $\widetilde{N}^{(p)}$ is defined in (\ref{eq:Ntilde}). 
We employ Corollary~\ref{cor:CaccHighInt} to the pair ($B_i$, $\widehat B_i)$ of concentric balls 
together with Lemma~\ref{lem:estH12D} for $t=1/2-\varepsilon/2$ and $N^2(U,F,f) \leq C_N \widetilde N^{(p)} (F,f)$ 
to obtain, for suitable $\gamma > 0$, 
\begin{align*}
\norm{\dbeta\nabla U}^2_{L^2_\alpha(B^+_{i})} \leq
  \gamma^{2p+1} R_i^{-2p+1-\varepsilon} p^{2p} \widetilde N^{(p)} (F,f). 
\end{align*}
Summation over $i$ (with very generous bounds for the
data $f$, $F$) and (\ref{eq:estimate-r_c}) provides 
\begin{align*}
\|r_{\mathbf{v}}^{p - 1/2 +\varepsilon} \dbeta \nabla U\|^2_{L^2_\alpha(\omega^+_{\mathbf{v}})}
   &\leq
       C_B^{2p-1+2\varepsilon}\sum_i R_i^{2p-1+2\varepsilon} \|\dbeta \nabla U\|^2_{L^2_{\alpha}(B_{i}^+)}
  \\ & \leq 
      \gamma^{2p+1} C_B^{2p+1} p^{2p}\bigg( \sum_i R_i^{\varepsilon} \biggr) \widetilde N^{(p)} (F,f)
\\ & \leq 
       C_\varepsilon (\gamma C_B)^{2p+1}p^{2p} \bigg\{ \|f\|^2_{H^{1}(\Omega)} + \|F\|^2_{L^2_{-\alpha}(\R^2\times(0,H))}   
       \\ & \qquad +
       \sum_{j=2}^{p+1} p^{-2j}\left(\max_{\etam=j}\norm{\deta f}^2_{L^2(\Omega)} 
  +\max_{\etam=j-1} \norm{\deta F}^2_{L^2_{-\alpha}(\R^2\times(0,H))} \right)\bigg\},
\end{align*}
since $\sum_{i}R_i^{\varepsilon} \eqqcolon C_\varepsilon < \infty$ 
by Lemma~\ref{lemma:covering-omega_c}. 
Relabelling $\gamma C_B$ as $\gamma$ gives the result. 
\end{proof}

We continue with the more involved case of vertex-edge neighborhoods.
  \begin{lemma}[Weighted $H^p$-regularity in $\omega_{\mathbf{v}\mathbf{e}}$]
\label{lemma:regularity-omega_ce}
Let $\omegaeps>0$ be sufficiently small.
There exists $\gamma > 0$ depending only on $s$, $\omegaeps$, $\Omega$, and $H$,  and for  
any $\varepsilon \in (0,1)$, there exists $C_\varepsilon>0$ depending additionally on $\varepsilon$
such that the solution $U$ of (\ref{eq:minimization}) satisfies, for all $(\ppar,\pperp)\in \N^2_0$ with $p = \ppar + \pperp$ 
\begin{align*}
  &\norm{r_{\mathbf{e}}^{\pperp -1/2+\varepsilon} 
         r_{\mathbf{v}}^{\ppar+\varepsilon} D^{\pperp}_{x_\perp} D^{\ppar}_{x_\parallel} \nabla U}^2_{L^2_\alpha((\omegace^\omegaeps)^{H/4})}
  \\ &\quad \leq C_\varepsilon \gamma^{2p+1} p^{2p}
           \bigg[ \| f\|_{H^1(\Omega)}^2  + \|F\|^2_{L^2_{-\alpha}(\R^2\times (0,H))}   + \sum_{j=2}^{p+1} p^{-2j} \Big(\max_{\etam=j}\norm{\deta f}^2_{L^2(\Omega)} 
                  +\max_{\etam=j-1} \norm{\deta F}^2_{L^2_{-\alpha}(\R^2\times(0,H))}\Big)\bigg]. 
\end{align*}
\end{lemma}
\begin{proof}
As in the proof of Lemma~\ref{lemma:regularity-omega_c}, the case $p = 0$ follows from Lemma~\ref{lem:estH12D} 
and the estimates (\ref{eq:CUFf}), (\ref{eq:CUFf-simplified}) so that we may assume $p \in \N$.
By Lemma~\ref{lemma:subcovering-omega_ce}, for sufficiently small $\omegaeps$, 
there is a covering of $\omegace^{\omegaeps}$ by half-balls $(H_i)_{i \in {\mathbb N}}$ 
with corresponding stretched half-balls $(\widehat H_i)_{i \in {\mathbb N}}$ and 
intermediate half-balls $(\widetilde H_i)_{i \in {\mathbb N}}$ 
such that each $H_i$ is covered by balls 
${\mathcal B}_i\coloneqq \{B_{ij}\,|\, j \in {\mathbb N}\}$ with the stretched balls $\widehat B_{ij}$ satisfying a finite overlap condition 
and being contained in $\widetilde H_i$. 
We abbreviate the radii of the half-balls $\widehat H_i$ and the balls $\widehat B_{ij}$ by $R_i$ and $R_{ij}$ respectively. 
We note that the half-balls $\widehat H_i$ and the balls $\widehat B_{ij}$ satisfy for all $i$, $j$: 
\begin{align}
\label{eq:lemma:regularity-omega_ce-10}
\forall x \in \widehat H_i: \qquad C_B^{-1} R_i \leq r_{\mathbf v}(x) \leq C_B R_i, \\ 
\label{eq:lemma:regularity-omega_ce-20}
\forall x \in \widehat B_{ij}: \qquad C_B^{-1} R_{ij} \leq r_{\mathbf e}(x) \leq C_B R_{ij}  
\end{align}
for some $C_B>0$ depending only on $\omegace^{\omegaeps}$. 
For convenience, we assume that $R_i \leq 1/2$ for all $i$ and  
hence $R_{ij} \leq 1/2$ for all $i$, $j$. 

Let ${\ppar}$, ${\pperp} \in \N_0$. Since the balls $(B_{ij})_{i,j \in {\mathbb N}}$ cover $\omegace^{\omegaeps}$, 
we estimate using 
(\ref{eq:lemma:regularity-omega_ce-10}), (\ref{eq:lemma:regularity-omega_ce-20}) 
\begin{align}\label{eq:loctemp1}
 &\norm{r_{\mathbf{e}}^{{\pperp} - 1/2+\varepsilon/2} 
        r_{\mathbf{v}}^{{\ppar}+\varepsilon} D^{\pperp}_{x_\perp} D^{\ppar}_{x_\parallel}  \nabla U}_{L^2_\alpha((\omegace^{\omegaeps})^{H/4})}^2 
        \nonumber
   \\
&\qquad \leq C_B^{2 {\pperp} - 1 +  \varepsilon + 2{\ppar} + 2\varepsilon} \sum_{i,j} R_i^{2{\ppar}+2\varepsilon} R_{ij}^{2 {\pperp}-1+\varepsilon}
  \norm{D^{\pperp}_{x_\perp} D^{\ppar}_{x_\parallel}  \nabla U}^2_{L^2_\alpha(B_{ij}^{H/4})}. 
\end{align}
With the constant $\gamma>0$ from Corollary~\ref{cor:CaccHighInt}, we abbreviate
\begin{align*}
  \Npperpij(F, f) & \coloneqq 
 \sum_{n=1}^{\pperp}(\gamma\pperp)^{-2n}\left(\max_{\etam=n}\norm{\deta D_{x_\parallel}^{\ppar} f}^2_{L^2(\widehat B_{ij})}
  +\max_{\etam=n-1} \norm{\deta D_{x_\parallel}^{\ppar} F}^2_{L^2_{-\alpha}(\widehat B_{ij} \times (0,H))}\right), \\
  \Npperpi(F, f) & \coloneqq 
 \sum_{n=1}^{\pperp}(\gamma\pperp)^{-2n}\left(\max_{\etam=n}\norm{\deta D_{x_\parallel}^{\ppar} f}^2_{L^2(\widetilde H_{i})}
  +\max_{\etam=n-1} \norm{\deta D_{x_\parallel}^{\ppar} F}^2_{L^2_{-\alpha}(\widetilde H_{i} \times (0,H))}\right).
\end{align*}
Applying the interior Caccioppoli-type estimate (Corollary~\ref{cor:localYCaccioppoli}) for the pairs 
$(B_{ij}\times (0,H/4),\widehat B_{ij}\times (0,H/2))$ and the function $D^{\ppar}_{x_\parallel}U$ (noting that this function satisfies \eqref{eq:extension2D} with data $D^{\ppar}_{x_\parallel}f$, $D^{\ppar}_{x_\parallel}F$) provides (we also use $R_i \leq 1/2 \leq 1$)
\begin{align}\label{eq:loctemp1a}
& \norm{ D^{\pperp}_{x_\perp} \nabla  D^{\ppar}_{x_\parallel} U}_{L^2_\alpha(B_{ij}^{H/4})}^2  \leq 
 2^{{\pperp}}\max_{\betam = {\pperp}} \norm{ \dbeta \nabla D^{\ppar}_{x_\parallel} U}_{L^2_\alpha(B_{ij}^{H/4})}^2 
  \\
\nonumber 
& \qquad \qquad \leq  (\sqrt{2}\gamma \pperp)^{2{\pperp}}R_{ij}^{-2{\pperp}}\bigg(
\norm{\nabla D_{x_\parallel}^{\ppar} U}^2_{L^2_\alpha(\widehat B_{ij}^{H/2})} + R_{ij}^{2}  \Npperpij(F, f)
\bigg) \\
\nonumber 
& \qquad \qquad \stackrel{(\ref{eq:lemma:regularity-omega_ce-20})}{\leq} C_B^{1-\varepsilon} (\sqrt{2}\gamma \pperp)^{2{\pperp}}R_{ij}^{-2{\pperp}+1-\varepsilon}\bigg(
\norm{r^{-1/2+\varepsilon/2}_{{\mathbf e}}\nabla D_{x_\parallel}^{\ppar} U}^2_{L^2_\alpha(\widehat B_{ij}^{H/2})} + R_{ij}^{1+\varepsilon}  \Npperpij(F, f)
\bigg) . 
\nonumber 
\end{align}
Inserting this in (\ref{eq:loctemp1}), summing over all $j$, and  
using the finite overlap property as well as $R_{ij} \leq R_i$ yields 
\begin{align}
\nonumber 
 &\norm{r_{\mathbf{e}}^{{\pperp} - 1/2+\varepsilon/2} 
        r_{\mathbf{v}}^{{\ppar}+\varepsilon} D^{\pperp}_{x_\perp} D^{\ppar}_{x_\parallel} \nabla U}_{L^2_\alpha((\omegace^{\omegaeps})^{H/4})}^2 
\\
\label{eq:loctmp10}
&\lesssim C_B^{2\pperp+2\ppar+2\varepsilon}
(\sqrt{2}\gamma \pperp)^{2\pperp} 
\sum_i R^{2\ppar+ 2\varepsilon}_i 
\left( 
\|r_{\mathbf e}^{-1/2+\varepsilon/2} \nabla D^{\ppar}_{x_\parallel} U\|^2_{L^2_\alpha(\widetilde H_i^{H/2})} + R_i^{1+\varepsilon} \Npperpi(F,f)
\right),
\end{align}
with the implied constant reflecting the overlap constant. 
Using again $R_i \leq 1$, 
we estimate the sum over the $\Npperpi(F,f)$ (generously) 
by
\begin{equation*}
\sum_i R^{2\ppar + 2\varepsilon}_i R^{1+\varepsilon}_i \Npperpi(F,f) \leq C  \sum_{n=1}^{\pperp} (\gamma \pperp)^{-2n} 
\left(\max_{|\eta| = n} \|\partial^\eta_x D^{\ppar}_{x_\parallel} f\|^2_{L^2(\Omega)} + 
      \max_{|\eta| = n-1} \|\partial^\eta_x D^{\ppar}_{x_\parallel} F\|^2_{L^2_{-\alpha}(\Omega \times (0,H))} \right). 
\end{equation*}
The term involving $\|r^{-1/2+\varepsilon/2}_{{\mathbf e}} \nabla D^{\ppar}_{x_\parallel} U\|^2_{L^2_\alpha(\widetilde H_i^{H/2})}$ in 
(\ref{eq:loctmp10}) is treated with Lemma~\ref{lem:estH12D} for the case $\ppar = 0$ and Lemma~\ref{lem:localhighregularity2D} for $\ppar > 0$. 
Considering first the case $\ppar = 0$, we estimate using the finite overlap property of the half-balls $\widehat H_i$ and $r_{\partial\Omega} \leq r_{\mathbf e}$ 
\begin{equation*}
\sum_i R^{2\ppar + 2\varepsilon}_i \|r^{-1/2+\varepsilon/2}_{{\mathbf e}} \nabla D^{\ppar}_{x_\parallel} U \|^2_{L^2_{\alpha}(\widetilde  H_i^{H/2})} 
\stackrel{\text{finite overlap}, \ppar = 0}{\lesssim} \|r^{-1/2+\varepsilon/2}_{\partial\Omega} \nabla U\|^2_{{L^2_{\alpha}(\Omega^+)}} \stackrel{\text{L.~\ref{lem:estH12D}}}{\lesssim}
N^2(U,F,f). 
\end{equation*}
For $\ppar > 0$, we use Lemma~\ref{lem:localhighregularity2D}. To that end, we select, for each $i\in\N$, a cut-off function $\eta_i \in C^\infty_0(\R^2)$ 
with $\operatorname{supp} \eta_i \cap \Omega \subset \widehat H_i$ and $\eta_i \equiv 1$ on $\widetilde H_i$ and a cut-off function $\eta_y \in C_0^\infty(-H,H)$ with $\eta_y \equiv 1$ on $(-H/2,H/2)$. Applying Lemma~\ref{lem:localhighregularity2D} 
with $t = 1/2-\varepsilon/2$ there and using the finite overlap property we get for $\widetilde U^{(\ppar)}_i\coloneqq \eta_i\eta_y D^{\ppar}_{x_\parallel} U$ and 
$\widetilde{N}^{(\ppar)}(F,f)$ from (\ref{eq:Ntilde}) 
\begin{align*}
& \sum_i R^{2\ppar + 2\varepsilon}_i \|r^{-1/2+\varepsilon/2}_{{\mathbf e}} \nabla D^{\ppar}_{x_\parallel} U \|^2_{L^2_{\alpha}(\widetilde H_i^{H/2})}
\leq 
\sum_i R^{2\ppar + 2\varepsilon}_i \|r^{-1/2+\varepsilon/2}_{\partial\Omega} \nabla \widetilde{U}^{(\ppar)}_i \|^2_{L^2_{\alpha}(\widetilde H_i^{H/2})} \\
& \lesssim \sum_i R^{2\ppar+2\varepsilon-2 \ppar - 1 + 2(1/2-\varepsilon/2)}_i (\gamma \ppar)^{2\ppar} (1+ \gamma \ppar) \widetilde{N}^{(\ppar)}(F,f)
 \lesssim (\gamma \ppar)^{2\ppar} (1+ \gamma \ppar) \widetilde{N}^{(\ppar)}(F,f);
\end{align*}
here, we used that $\sum_i R^{\varepsilon}_i < \infty$ by Lemma~\ref{lemma:covering-omega_ce}. 

Combining the above estimates we have shown the existence of $C\geq 1$ independent of
$p = \ppar + \pperp$ such that 
\begin{align*}
 &\norm{r_{\mathbf{e}}^{\pperp - 1/2+\varepsilon/2} 
        r_{\mathbf{v}}^{\ppar+\varepsilon} D^{\pperp}_{x_\perp} D^{\ppar}_{x_\parallel} \nabla U}_{L^2_\alpha((\omegace^{\omegaeps})^{H/4})}^2 
  \\ & \qquad
       \leq C^{2p+1}\left[\pperp^{2\pperp}\ppar^{2\ppar+1}\tNppar(F, f) + \sum_{n=1}^{\pperp}\pperp^{2\pperp-2n} \biggl(\max_{\etam=n}\norm{\deta D_{x_\parallel}^{\ppar} f}^2_{L^2(\Omega)} 
  +\max_{\etam=n-1} \norm{\deta D_{x_\parallel}^{\ppar} F}^2_{L^2_{-\alpha}(\R^2\times (0,H))}\biggr)\right].
\end{align*}
For $\pperp \ge 1$, we estimate with $\pperp \leq p$ 
\begin{align*}
\sum_{n=1}^{\pperp} \pperp^{2(\pperp - n)} \max_{|\eta| = n} \|\partial^\eta_x D^{\ppar}_{x_\parallel} f\|^2_{L^2(\Omega)}
\leq \sum_{n=1}^{\pperp} p^{2(\pperp - n)} \max_{|\eta| = n} \|\partial^\eta_x D^{\ppar}_{x_\parallel} f\|^2_{L^2(\Omega)}
\leq \sum_{j=\ppar+1}^{p} p^{2(p - j)} \max_{|\eta| = j} \|\partial^\eta_x f\|^2_{L^2(\Omega)}
\end{align*}
and analogously for the sum over the terms $\max_{|\eta| = n-1} \|\partial^\eta_x D^{\ppar}_{x_\parallel} F\|^2_{L^2_{-\alpha}(\R^2 \times (0,H))}$. 
Also by similar arguments, we estimate $\ppar^{2\ppar} \tNppar(F,f) \leq C p^{2\ppar} \tNp(F,f)$. 
Using $\ppar+\pperp =  p$ as well as 
$|\Dpar^{\ppar} v | \leq 2^{\ppar/2}\max_{\betam=\ppar} |\dbeta v|$ completes the proof of the vertex-edge case 
in view of the definition of $\tNp(F,f)$ from (\ref{eq:Ntilde}), $r_{\mathbf{e}}^\varepsilon \lesssim r_{\mathbf{e}}^{\varepsilon/2}$, and by suitably selecting $\gamma$.
\end{proof}

\begin{lemma}[Weighted $H^p$-regularity in $\omega_{\mathbf{e}}$]
\label{lemma:regularity-omega_e}
Given $\omegaeps > 0$ and ${\mathbf e} \in {\mathcal E}$, 
there is $\gamma$ depending only on $s$, $\Omega$, $H$,
and $\omega_{\mathbf e} = \omega^{\omegaeps}_{{\mathbf e}}$ such that 
for every $\varepsilon \in (0,1)$ there is $C_\varepsilon>0$ depending additionally on $\varepsilon$ 
such that the solution $U$ of 
(\ref{eq:minimization})
satisfies, for all $(\ppar,\pperp) \in \N^2_0$ with $p = \ppar+\pperp $
\begin{align*}
  &\norm{r_{\mathbf{e}}^{\pperp -1/2+\varepsilon} D^{\pperp}_{x_{\perp}}D^{\ppar}_{x_{\ppar}}   \nabla U}^2_{L^2_\alpha((\omega^{\xi}_{\mathbf{e}})^{H/4})}
  \\ &\ \leq C_\varepsilon \gamma^{2p} p^{2p}\!\left[ \| f\|_{H^1(\Omega)}^2  + \|F\|^2_{L^2_{-\alpha}(\R^2\times (0,H))}   + \sum_{j=2}^{p+1} p^{-2j} \Big(\max_{\etam=j}\norm{\deta f}^2_{L^2(\Omega)} 
  +\max_{\etam=j-1} \norm{\deta F}^2_{L^2_{-\alpha}(\R^2\times(0,H))}\Big) \right].
\end{align*}
\end{lemma}
\begin{proof}
The proof is essentially identical to the case $\ppar = 0$ in the proof of Lemma~\ref{lemma:regularity-omega_ce} using a covering of $\omega_{\mathbf{e}}$ 
analogous to the covering of $\omega_{\mathbf v}$ given in
Lemma~\ref{lemma:covering-omega_c} that is refined towards ${\mathbf e}$ rather
than ${\mathbf v}$, see Figure \ref{fig:covering-edge}. 
\end{proof}
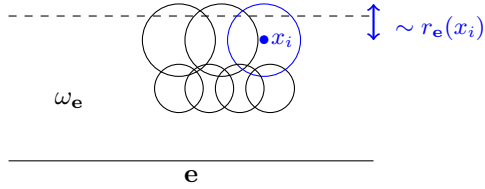
\begin{figure}[ht]
\begin{center} 
\begin{tikzpicture}[scale=0.8]
\draw (3,0) node[below]{$\mathbf{e}$}; 
  \draw[-] (0, -0) -- (6, 0);
  \draw[dashed] (0,2.4) -- (6,2.4); 
  \draw (1,1) node {$\omega_{\mathbf{e}}$}; 

\newcommand{\cirs}{0.6} 
\draw (3.5,2) circle (\cirs);
\draw[blue] (4.2,2) circle (\cirs);
\draw[blue] (4.2,2) node {\footnotesize \textbullet};
\draw[blue] (4.5,2) node {\small $x_{i}$} ; 

    \draw[blue,thick,<->] (6, 2) -- (6, 2+\cirs);
    \draw[blue] (6.2,2.2) node[right] {\small $\sim r_{\mathbf{e}}(x_i)$};
\draw (2.8,2) circle (\cirs);
\newcommand{\cirss}{0.4}
\draw (4.3,1.2) circle (\cirss);
\draw (3.8,1.2) circle (\cirss);
\draw (3.3,1.2) circle (\cirss);
\draw (2.8,1.2) circle (\cirss);
\end{tikzpicture}
\end{center}
\caption{\label{fig:covering-edge} Covering of  edge-neighborhoods
  $\omega_{\mathbf{e}}$. 
  }
\end{figure}

  \begin{remark} \label{remark:omegace-decomp}
  The assumption that $\omegaeps$ is sufficiently small in Lemma
  \ref{lemma:regularity-omega_ce} can be dropped (as long as $\omegace$ is
  well defined, as per Section \ref{sec:WgtAnRegR2}).
  Indeed, for all $\omegaeps_1, \omegaeps_2$ such that $\omegaeps_1\geq \omegaeps_2>0$ there exists $\omegaeps_3\geq
  \omegaeps_2$ such that
  \begin{equation}
    \label{eq:omegace-decomp}
   \omegace^{\omegaeps_1}  \subset \left( \omegace^{\omegaeps_2} \cup \omegac^{\omegaeps_3}\cup \omegae^{\omegaeps_3}\right).
  \end{equation}
  In addition, there exists a constant $C_{\omegaeps_3} > 0$ that depends only on
  $\omegaeps_3$ and $\varepsilon$ such that
  \begin{equation}
    \label{eq:omegac-decomp-equiv}
    \begin{aligned}
    \|r_{\mathbf{e}}^{\pperp-1/2+\varepsilon} r_{\mathbf{v}}^{\ppar+\varepsilon}D^{\pperp}_{x_{\perp}}D^{\ppar}_{x_{\ppar}}\nabla U\|_{L^2_{\alpha}((\omegac^{\omegaeps_3})^+)}^2
    & \leq 2^{p} \max_{\betam = p}
    \|r_{\mathbf{e}}^{\pperp-1/2+\varepsilon} r_{\mathbf{v}}^{\ppar+\varepsilon}\dbeta \nabla U\|_{L^2_{\alpha}((\omegac^{\omegaeps_3})^+)}^2
    \\ & \leq
    C_{\omegaeps_3}^{p+1}
    \max_{\betam = p}
    \|r_{\mathbf{v}}^{p-1/2+\varepsilon} \dbeta \nabla U\|_{L^2_{\alpha}((\omegac^{\omegaeps_3})^+)}^2
    \end{aligned}
  \end{equation}
  and that 
  \begin{equation}
    \label{eq:omegae-decomp-equiv}
    \|r_{\mathbf{e}}^{\pperp-1/2+\varepsilon} r_{\mathbf{v}}^{\ppar+\varepsilon}D^{\pperp}_{x_{\perp}}D^{\ppar}_{x_{\ppar}}\nabla U\|_{L^2_{\alpha}((\omegae^{\omegaeps_3})^+)}^2
    \leq C_{\omegaeps_3}^{p+1}
    \norm{r_{\mathbf{e}}^{\pperp -1/2+\varepsilon} D^{\pperp}_{x_{\perp}}D^{\ppar}_{x_{\ppar}}   \nabla U}^2_{L^2_{\alpha}((\omegae^{\omegaeps_3})^+)}.
  \end{equation}
  Given $\omegaeps_1>0$, bounds in $\omegace^{\omegaeps_1}$ can therefore be derived by choosing
  $\omegaeps_2$ such that Lemma \ref{lemma:regularity-omega_ce} holds in
  $\omegace^{\omegaeps_2}$, exploiting the decomposition
  \eqref{eq:omegace-decomp}, using Lemmas
  \ref{lemma:regularity-omega_c} and \ref{lemma:regularity-omega_ce} in
  $\omegac^{\omegaeps_3}$ and $\omegae^{\omegaeps_3}$, respectively, and
  concluding with \eqref{eq:omegac-decomp-equiv} and \eqref{eq:omegae-decomp-equiv}.
\eremk
  \end{remark}

\subsection{Proof of Theorem~\ref{thm:mainresult} -- weighted $H^p$ regularity for fractional PDE \eqref{eq:intro-eq}}
In order to obtain regularity estimates for the solution $u$ of $(-\Delta)^s u = f$, we have to take the trace 
$y\rightarrow 0$ in the weighted $H^p$-estimates for the Caffarelli-Silvestre extension problem provided by the previous subsection.

\begin{proposition}\label{pro:mainresultOLD}
Under the hypotheses of Theorem~\ref{thm:mainresult}, there exists a constant 
$\gamma>0$ depending only on $\gamma_f$, $s$, and $\Omega$ such that for 
every $\varepsilon>0$ there exists $C_\varepsilon>0$  
(depending only on $\varepsilon$ and $\Omega$) such that the following holds: 
\begin{enumerate}[(i)]
\item for all $p \in \N$ there holds 
\begin{subequations}
\label{eq:analytic-u-ceOLD}
\begin{equation}  \label{eq:analytic-u-ce1OLD}
\norm{r_{\mathbf{e}}^{-1/2 + \varepsilon} r_{\mathbf{v}} ^{p-s+\varepsilon} D^{p}_{x_\parallel} u}_{L^2(\omega_{\mathbf{v}\mathbf{e}})} 
\leq 
C_{\varepsilon}\gamma^{p+1} p^p.
\end{equation}
\item
For all $\ppar \in \N_0$, $\pperp \in \N$ with $\ppar+\pperp=p$ there holds 
\begin{equation}  \label{eq:analytic-u-ce2OLD}
\norm{r_{\mathbf{e}} ^{\pperp-1/2-s + \varepsilon} r_{\mathbf{v}}^{\ppar+\varepsilon} 
    D^{\pperp}_{x_\perp} D^{\ppar}_{x_\parallel} u}_{L^2(\omega_{\mathbf{v}\mathbf{e}})} 
\leq C_{\varepsilon} \gamma^{p+1} p^p.
\end{equation}
\end{subequations}
\item 
For all $\beta\in \N_0^2$ with $\betam = p \ge 1$ and all $\ppar \in \N_0$, $\pperp \in \N$ with $\ppar + \pperp = p \ge 1$ there holds 
\begin{align}
  \label{eq:analytic-u-cOLD}
 \norm{r_{\mathbf{v}}^{p-1/2-s+\varepsilon} \dbeta u }_{L^2(\omega_{\mathbf{v}})} \leq C_{\varepsilon} \gamma^{p+1}p^p, \\
  \label{eq:analytic-u-eOLD}
 \norm{r_{\mathbf{e}}^{\pperp-1/2-s+\varepsilon} D^{\pperp}_{x_\perp} D^{\ppar}_{x_{\parallel}} u }_{L^2(\omega_{\mathbf{e}})} \leq C_{\varepsilon}\gamma^{p+1}p^p.
\end{align}
\item For $\ppar \in \N$, we have 
\begin{align}
  \label{eq:analytic-u-fOLD}
 \norm{r_{\mathbf{e}}^{-1/2+\varepsilon} D^{\ppar}_{x_{\parallel}} u }_{L^2(\omega_{\mathbf{e}})} \leq C_{\varepsilon}\gamma^{p+1}p^p.
\end{align}
\item
For all $\beta\in \N_0^2$ there holds with $\betam = p$
\begin{equation}
  \label{eq:analytic-u-intOLD}
 \norm{\dbeta u }_{L^2(\Omega_{\rm int})} \leq \gamma^{p+1}p^p.
\end{equation}
\end{enumerate}
\end{proposition}
\begin{proof}

We only show the estimates \eqref{eq:analytic-u-ce1OLD} and \eqref{eq:analytic-u-ce2OLD} using Lemma~\ref{lemma:regularity-omega_ce}.
The bounds \eqref{eq:analytic-u-cOLD} 
(using Lemma~\ref{lemma:regularity-omega_c}) 
and \eqref{eq:analytic-u-eOLD}, \eqref{eq:analytic-u-fOLD} 
(using Lemma~\ref{lemma:regularity-omega_e}) follow with identical arguments. 
The bound in $\Omega_{\rm int}$ follows directly from the interior Caccioppoli inequality, Corollary~\ref{cor:CaccHighInt}, 
and a trace estimate as below. 
(Note that the case $|\beta| = 0$ follows directly from the energy estimate
$\|u\|_{L^2(\Omega_{\rm int})} \leq \|u\|_{\widetilde{H}^s(\Omega)} \leq C \|f\|_{H^{-s}(\Omega)}$.)
\medskip

Due to Lemma~\ref{lemma:regularity-omega_ce}, applied with $F=0$, and the assumption \eqref{eq:analyticdata} on the data $f$, there
exists a constant $C>0$ such that for all $(\qperp,\qpar)\in \N^2_0$ we have with $q = \qperp+\qpar \in \N_{0}$ 
\begin{equation}
\label{eq:analytic-nablaU-ce}
\norm{r_{\mathbf{e}}^{\qperp -1/2+\varepsilon}r_{\mathbf{v}}^{\qpar +\varepsilon} 
D^{\qperp}_{x_\perp} D^{\qpar}_{x_\parallel} \nabla U}_{L^2_\alpha(\omegace^{H/4})}^2 
\leq C^{2q+1}q^{2q}.
\end{equation}
The last step of the proof of \cite[Lem.~3.7]{KarMel19} gives the multiplicative trace estimate 
\begin{align}
\label{eq:L3.7-KarMel19}
 \abs{V(x,0)}^2 
\leq 
C_{\mathrm{tr}} \left( 
\norm{V(x,\cdot)}_{L^2_\alpha(\Rpos)}^{1-\alpha}\norm{\partial_y V(x,\cdot)}_{L^2_\alpha(\Rpos)}^{1+\alpha} 
+
\|V(x,\cdot)\|^2_{L^2_\alpha(\Rpos)}
\right), 
\;\; x\in \Omega, 
\end{align}
where, for univariate $v:\Rpos\rightarrow\R$, 
we write $\|v\|^2_{L^2_\alpha(\Rpos)}\coloneqq \int_{y=0}^\infty y^\alpha |v(y)|^2\,dy$.
Applying this estimate to $\eta_y V$ with a cut-off function $\eta_y \in C^{\infty}_0(-H/4,H/4)$ satisfying $\eta_y = 1$ on $(-H/8,H/8)$ shows that in (\ref{eq:L3.7-KarMel19})
$\R_+$ can be replaced by $(0,H/4)$ at the expense of a constant depending additionally on $H$.

We have $p = \pperp + \ppar \ge 1$. 
Suppose first $\pperp\geq 1$ and $\ppar \ge 0$. 
Using the trace estimate (\ref{eq:L3.7-KarMel19}) with $V = D^{\pperp}_{x_\perp} D^{\ppar}_{x_\parallel} U$ 
and additionally multiplying with the corresponding weight (using that $\alpha = 1-2s$) 
provides
\begin{align*}
 &r_{\mathbf{e}}^{2\pperp -1 -2s+2\varepsilon} r_{\mathbf{v}}^{2\ppar +
   2\varepsilon}\abs{D^{\pperp}_{x_\perp} D^{\ppar}_{x_\parallel} U(x,0)}^2
\\ &\qquad\quad\leq C_{\mathrm{tr}} 
 \norm{r_{\mathbf{e}}^{\pperp -3/2 +\varepsilon} r_{\mathbf{v}}^{\ppar+\varepsilon}\nabla D^{\pperp-1}_{x_\perp}D^{\ppar}_{x_\parallel} U(x,\cdot)}_{L^2_\alpha(0,H/4)}^{1-\alpha}
 \norm{r_{\mathbf{e}}^{\pperp -1/2 +\varepsilon} r_{\mathbf{v}}^{\ppar + \varepsilon}D_{x_\perp}^{\pperp} D^{\ppar}_{x_\parallel}\nabla  U(x,\cdot)}_{L^2_\alpha(0,H/4)}^{1+\alpha} \\
 &\qquad\quad\quad+C_{\mathrm{tr}} 
 \norm{r_{\mathbf{e}}^{\pperp -1/2-s +\varepsilon} r_{\mathbf{v}}^{\ppar+\varepsilon}\nabla D^{\pperp-1}_{x_\perp}D^{\ppar}_{x_\parallel} U(x,\cdot)}_{L^2_\alpha(0,H/4)}^{2},
\end{align*}
where we have also used the fact that $(D_{x_\perp} v)^2  =
  (\mathbf{e}_\perp \cdot \nabla_x v)^2\leq | \nabla_x
  v|^2$ for all sufficiently smooth functions $v$.
Integration over $\omega_{\mathbf{v}\mathbf{e}}$ together with $r_{\mathbf{e}}^{-s}\lesssim r_{\mathbf{e}}^{-1}$ gives 
\begin{align*}
  &\norm{r_{\mathbf{e}}^{\pperp -1/2 -s+\varepsilon} r_{\mathbf{v}}^{\ppar + \varepsilon}D^{\pperp}_{x_\perp} D^{\ppar}_{x_\parallel}  u}_{L^2(\omega_{\mathbf{v}\mathbf{e}})}^2 \\
 &\qquad\quad \leq C_{\mathrm{tr}} 
  \norm{r_{\mathbf{e}}^{\pperp -3/2 +\varepsilon} r_{\mathbf{v}}^{\ppar+\varepsilon}D^{\pperp-1}_{x_\perp}D^{\ppar}_{x_\parallel} \nabla U}_{L^2_\alpha(\omega_{\mathbf{v}\mathbf{e}}^{H/4})}^{1-\alpha}
 \norm{r_{\mathbf{e}}^{\pperp -1/2 +\varepsilon} r_{\mathbf{v}}^{\ppar + \varepsilon}D_{x_\perp}^{\pperp} D^{\ppar}_{x_\parallel}\nabla  U}_{L^2_\alpha(\omega_{\mathbf{v}\mathbf{e}}^{H/4})}^{1+\alpha} \\
  &\qquad\quad\quad + C_{\mathrm{tr}} \norm{r_{\mathbf{e}}^{\pperp -1/2-s +\varepsilon} r_{\mathbf{v}}^{\ppar+\varepsilon}D^{\pperp-1}_{x_\perp}D^{\ppar}_{x_\parallel} \nabla U}_{L^2_\alpha(\omega_{\mathbf{v}\mathbf{e}}^{H/4})}^{2} \\
 &\qquad\quad\stackrel{\eqref{eq:analytic-nablaU-ce}}{\leq} C_{\mathrm{tr}}  (C^{2p-1}(p-1)^{2(p-1)})^{(1-\alpha)/2} (C^{2p+1}p^{2p})^{(1+\alpha)/2} + CC_{\mathrm{tr}}C^{2p-1}(p-1)^{2(p-1)}  \\ 
 &\qquad\quad\leq  C_{\mathrm{tr}}C^{2p+1+\alpha}p^{2p+\alpha}+C_{\mathrm{tr}}C^{2p-1}p^{2p} \leq  \gamma^{2p+1}p^{2p}
\end{align*}
for suitable $\gamma > 0$,
which is estimate \eqref{eq:analytic-u-ce2OLD}.
If $\pperp=0$, then $\ppar \ge 1$, and we have instead
\begin{align*}
 & \norm{r_{\mathbf{e}}^{-1/2 +\varepsilon} r_{\mathbf{v}}^{\ppar -s + \varepsilon} D^{\ppar}_{x_\parallel}  u}_{L^2(\omega_{\mathbf{v}\mathbf{e}})}^2 \\
  &\qquad\quad \leq C_{\mathrm{tr}} 
  \norm{r_{\mathbf{e}}^{ -1/2 +\varepsilon} r_{\mathbf{v}}^{\ppar-1+\varepsilon}\nabla D^{\ppar-1}_{x_\parallel} U}_{L^2_\alpha(\omega_{\mathbf{v}\mathbf{e}}^{H/4})}^{1-\alpha}
 \norm{r_{\mathbf{e}}^{ -1/2 +\varepsilon} r_{\mathbf{v}}^{\ppar + \varepsilon} D^{\ppar}_{x_\parallel}\nabla  U}_{L^2_\alpha(\omega_{\mathbf{v}\mathbf{e}}^{H/4})}^{1+\alpha} \\
 &\qquad\quad\quad + C_{\mathrm{tr}} 
  \norm{r_{\mathbf{e}}^{ -1/2 +\varepsilon} r_{\mathbf{v}}^{\ppar-s+\varepsilon}\nabla D^{\ppar-1}_{x_\parallel} U}_{L^2_\alpha(\omega_{\mathbf{v}\mathbf{e}}^{H/4})}^{2}.
\end{align*}
Again, inserting \eqref{eq:analytic-nablaU-ce} into the right-hand side and proceeding similarly as above 
proves \eqref{eq:analytic-u-ce1OLD}.
\end{proof}

We now apply Proposition~\ref{pro:mainresultOLD} to show our main result.

\begin{proof}[Proof of Theorem~\ref{thm:mainresult}]
Proposition~\ref{pro:mainresultOLD} already covers most of the statements in Theorem~\ref{thm:mainresult}. Only some lowest-order cases $p=0$ or $\pperp=0$ are missing.
  We consider the three inequalities 
  \eqref{eq:analytic-u-c-all}, 
  \eqref{eq:analytic-u-e-all}, and
  \eqref{eq:analytic-u-ce-all}
separately by using a Hardy inequality and then appealing to Proposition~\ref{pro:mainresultOLD}.
  \paragraph{\textbf{Proof of \eqref{eq:analytic-u-c-all}}}
  Equation \eqref{eq:analytic-u-c-all} with $p=0$ follows from the weighted
 Hardy inequality \cite[Lem.~{7.1.3}]{Kozlov1997}, which provides 
   \begin{equation*}
     \| r_{\mathbf{v}}^{-1/2-s+\varepsilon} u \|_{L^2(\omega_\mathbf{v})} \leq C_{\mathrm{H},1} \| 
 r_{\mathbf{v}}^{1/2-s+\varepsilon} \nabla u \|_{L^2(\omega_\mathbf{v})} 
\stackrel{\text{Prop.~\ref{pro:mainresultOLD}}}{<} \infty.  
   \end{equation*}
  \paragraph{\textbf{Proof of \eqref{eq:analytic-u-e-all}}}
Let $(x_{\perp}, x_{\parallel})$ be the coordinate system associated with edge ${\mathbf e}$. For $\mu$, $\xi > 0$ sufficiently small 
and an interval $I_\mu$ of length $\mu$ consider 
  \begin{equation*}
    \omegae^{\xi} \subseteq \{(x_\perp, x_{\parallel}) : x_{\parallel}\in I_\mu, x_{\perp}\in (0,\xi^2)\} =: \widetilde\omega_{\mathbf{e}}^{\xi,\mu}.
  \end{equation*}
The interval $I_\mu$ is chosen such that $\omega^\xi _{{\mathbf e}} \subset \widetilde\omega_{\mathbf{e}}^{\xi,\mu}$  and 
$\widetilde \omega^{\xi,\mu}_{{\mathbf e}}$ stays away from the vertices ${\mathcal V}$ and the edges ${\mathcal E}\setminus\{{\mathbf e}\}$
so that the assertions of Proposition~\ref{pro:mainresultOLD} still hold for $\widetilde{\omega}^{\xi,\mu}_{{\mathbf e}}$ -- cf.\ Remark~\ref{remark:omegace-decomp}. 
We will show \eqref{eq:analytic-u-e-all} for $\widetilde\omega_{{\mathbf e}}$ (dropping the superscripts $\xi$, $\mu$). 

  Let $\widetilde u$ be the
  function such that $\widetilde u(x_{\perp}, x_{\parallel}) = u(x_1,x_2 )$ in $\widetilde\omega_{\mathbf{e}}$.
  By Fubini-Tonelli's theorem, for almost all $x_{\parallel}\in I_\mu$, there holds
     \begin{equation}
       \label{eq:u-e-LinfL2}
  \bigg(  x_{\perp } \mapsto r_{\mathbf{e}}^{1/2-s+\epsilon}D_{x_\perp}(D^{\ppar}_{x_{\parallel}}\widetilde u )(x_{\perp}, x_{\parallel}) \bigg) \in L^2((0,\xi^2)).
     \end{equation}
The fundamental theorem of calculus, the Cauchy-Schwarz inequality, and 
\eqref{eq:u-e-LinfL2}, imply that, for almost all $x_\parallel\in I_\mu$, one has 
for $\varepsilon < s$ that 
$(D^{\ppar}_{x_\parallel}\widetilde u )(\cdot, x_{\parallel})\in C^{0, s-\varepsilon}([0,\xi^2])$. 
As $u \in \widetilde{H}^s(\Omega)$, we infer the pointwise equality 
$(D^{\ppar}_{x_\parallel}\widetilde u) (0, x_{\parallel})=0$ for almost all $x_{\parallel}$.
We can apply \cite[Lem.~{7.1.3}]{Kozlov1997} again, in one dimension: for
almost all $x_\parallel\in I_\mu$, there holds
\begin{equation*}
   \|
  r_{\mathbf{e}}^{-1/2-s+\varepsilon}(D^{\ppar}_{x_{\parallel}}\widetilde u )(\cdot, x_{\parallel}) 
  \|_{L^2((0, \xi^2))}
  \leq C_{\mathrm{H},2}
   \|
  r_{\mathbf{e}}^{1/2-s+\varepsilon}(D_{x_\perp}D^{\ppar}_{x_{\parallel}}\widetilde u )(\cdot, x_{\parallel}) 
  \|_{L^2((0,\xi^2))}.
\end{equation*}
Squaring and integrating over $x_\parallel\in I_\mu$ concludes the proof of \eqref{eq:analytic-u-e-all}.
\paragraph{\textbf{Proof of \eqref{eq:analytic-u-ce-all}}}
We use the same notation as in the previous part of the proof, but assume that the coordinate system $(x_1,x_2)$ and 
the coordinate system $(x_\perp,x_\parallel)$ associated with edge ${\mathbf e}$ satisfy $x_1 = x_\parallel$ and $x_2 = x_\perp$. 
Correspondingly, we assume  $I_\mu = (0,\mu)$. We introduce the equivalent vertex-edge neighborhood
\begin{equation*}
  \widetilde\omega_{\mathbf{ve}}^{\xi,\mu} = \{(x_\perp, x_\parallel) :  x_\parallel \in (0,\mu), x_{\perp} \in (0, \xi x_\parallel)\}.
\end{equation*}
We remark that in $\widetilde\omega_{\mathbf{ve}}$ there exists $c\geq 1$ such that for all $(x_\perp,
x_\parallel )\in \widetilde\omega_{\mathbf{ve}}$
\begin{equation}
  \label{eq:rv-xpar-equiv}
  x_{\parallel} \leq r_{\mathbf{v}}(x_{\parallel}, x_\perp) \leq c x_{\parallel}.
\end{equation}
We note $r_{{\mathbf e}}(x_\perp, x_\parallel) = x_\perp$. 
Hence, for almost all $x_\parallel\in (0,\mu)$, there holds
\begin{equation}
       \label{eq:u-ce-LinfL2}
  \bigg(  x_{\perp } \mapsto r_{\mathbf{e}}^{1/2-s+\varepsilon}(D_{x_\perp}(D^{\ppar}_{x_{\parallel}}\widetilde u ))(x_{\perp}, x_{\parallel}) \bigg) \in L^2((0, \xi x_\parallel)).
     \end{equation}
     By the same argument as above, it follows that, for almost all $x_\parallel \in (0, \mu)$, we have 
     $(D_{x_\parallel}^{\ppar}\widetilde u)(\cdot, x_\parallel)\in C^{0, s-\varepsilon}([0,
     \xi x_\parallel])$ and hence $(D_{x_\parallel}^{\ppar}\widetilde u)(0, x_\parallel)= 0$.
     Therefore,
     \cite[Lem.~{7.1.3}]{Kozlov1997} gives  for almost all $x_\parallel\in(0,\mu)$
     \begin{align*}
       \|
       r_{\mathbf{e}}^{-1/2-s+\varepsilon}(D^{\ppar}_{x_{\parallel}}\widetilde u) (\cdot, x_\parallel)
       \|_{L^2((0, \xi x_{\parallel}))}
       \leq C_{\mathrm{H},3}
       \|
       r_{\mathbf{e}}^{1/2-s+\varepsilon}(D_{x_\perp}D^{\ppar}_{x_{\parallel}}\widetilde u)(\cdot, x_\parallel) 
       \|_{L^2((0, \xi x_{\parallel}))},
     \end{align*}
     with a constant $C_{\mathrm{H},3}$ independent of $x_\parallel$.
     Multiplying by $r_{\mathbf{v}}^{\ppar+\varepsilon}$, squaring, integrating over $x_\parallel\in (0,\mu)$, and using \eqref{eq:rv-xpar-equiv}, we obtain
     \begin{align*}
       \|
       r_{\mathbf{e}}^{-1/2-s+\varepsilon}r_{\mathbf{v}}^{\ppar+\varepsilon}D^{\ppar}_{x_{\parallel}}\widetilde u 
       \|_{L^2(\widetilde\omega_{\mathbf{ve}})}
       \leq
       c^{\ppar+\varepsilon} C_{\mathrm{H},3}
       \|
       r_{\mathbf{e}}^{1/2-s+\varepsilon}r_{\mathbf{v}}^{\ppar+\varepsilon}D_{x_\perp}D^{\ppar}_{x_{\parallel}}\widetilde u 
       \|_{L^2(\widetilde\omega_{\mathbf{ve}})}.
     \end{align*} 
     This completes the proof except for the fact that the region $\omega_{{\mathbf v}{\mathbf e}}\setminus \widetilde{\omega}_{{\mathbf v}{\mathbf e}}$ 
is not covered yet. This region is treated with the observations of Remark~\ref{remark:omegace-decomp}.  
\end{proof}


%
\section{Conclusions}
\label{sec:Concl}
We briefly recapitulate the principal findings of the present paper, 
outline generalizations of the present results, and also indicate 
applications to the numerical analysis of finite element approximations of \eqref{eq:modelproblem}.
We established
analytic regularity of the solution $u$ in a scale of edge- and vertex-weighted Sobolev spaces
for the Dirichlet problem for the fractional Laplacian in a bounded polygon $\Omega\subset\R^2$ with
straight sides, and for forcing $f$ analytic in $\overline{\Omega}$.

While the analysis in Sections \ref{sec:LocTgReg} and \ref{sec:WghHpPolygon}
was developed at present in two spatial dimensions,
we emphasize that all parts of the proof can be extended to higher spatial dimension $d\geq 3$,
and polytopal domains $\Omega\subset \R^d$. Details shall be presented elsewhere.

Likewise, the present approach is also capable of handling nonconstant, analytic
coefficients similar to the setting considered (for the spectral fractional Laplacian)
in \cite{BMNOSS19}. 
Details on this extension of the present results, with the 
presently employed techniques, will also be developed in forthcoming work.

The weighted analytic regularity results obtained in the present paper can be used to establish 
\emph{exponential convergence rates} with the bound $C\exp(-b\sqrt[4]{N})$
on the error for suitable $hp$-Finite Element discretizations of \eqref{eq:modelproblem}, 
with $N$ denoting the number of degrees of freedom of the discrete solution in
$\Omega$. This will be proved in the follow-up work \cite{FMMS-hp}.
Importantly, as already observed in \cite{BMNOSS19},
achieving this exponential rate of convergence mandates
\emph{anisotropic mesh refinements} near the boundary $\partial\Omega$.
%
\appendix
\section{Localization of Fractional Norms}
\label{app:LocFrac}
The following elementary observation on localization of fractional norms 
was used in several places.
\begin{lemma}
\label{lemma:localization-fractional-norms}
Let $\eta \in C^\infty_0(B_R)$ for some ball $B_R\subset \Omega$ of radius $R$ and $s\in(0,1)$. 
Then,  
\begin{align}
\label{eq:lemma:localization-fractional-norms-10}
\|\eta f\|_{H^{-s}(\Omega)} & \leq C_{\rm loc} \|\eta\|_{L^\infty(B_R)} \|f\|_{L^2(B_R)},  \\
\label{eq:lemma:localization-fractional-norms-20}
  \|\eta f\|_{H^{1-s}(\Omega)} &
      \begin{multlined}[t][.5\textwidth]
        \leq \Cloctwo \big[\left( R^{s} \|\nabla \eta\|_{L^\infty(B_R)}  + (R^{s-1}+1) \|\eta\|_{L^\infty(B_R)}\right) \|f\|_{L^2(\Omega)}  \\
        + \|\eta\|_{L^\infty(B_R)} |f|_{H^{1-s}(\Omega)} \big], 
      \end{multlined}
\end{align}
where the constants $C_{\rm loc}$, $\Cloctwo$  depend only on $\Omega$ and $s$. 
\end{lemma}
\begin{proof}
\eqref{eq:lemma:localization-fractional-norms-10} follows directly from the embedding $L^2\subset H^{-s}$. For 
\eqref{eq:lemma:localization-fractional-norms-20}, we use the definition of the Slobodecki norm and the triangle inequality to write 
\begin{align*}
\abs{\eta f}_{H^{1-s}(\Omega)}^2 &= \int_{\Omega} \int_{\Omega} \frac{|\eta(x)f(x) - \eta(z)f(z)|^2}{\abs{x-z}^{d+2-2s}}\,dz\,dx \\
& \lesssim \int_{\Omega} \int_{\Omega} \frac{|\eta(x)f(x) - \eta(x)f(z)|^2}{\abs{x-z}^{d+2-2s}} \,dz\,dx 
+ \int_{\Omega} \int_{\Omega} \frac{|\eta(x)f(z) - \eta(z)f(z)|^2}{\abs{x-z}^{d+2-2s}}\,dz\,dx.
\end{align*}
The first term on the right-hand side can directly be estimated by $ \|\eta\|_{L^\infty(B_R)} |f|_{H^{1-s}(\Omega)}$.
 For the second term, we split the integration over $\Omega \times \Omega$ into four subsets, 
$B_{2R} \times B_{3R}$, 
$B_{2R}\times B_{3R}^c \cap\Omega$, 
$B_{2R}^c \cap\Omega \times B_R$, 
$B_{2R}^c \cap\Omega \times B_R^c \cap\Omega$; 
here, for simplicity we assume for the concentric balls $B_R \subset B_{2R} \subset B_{3R} \subset \Omega$, otherwise one has to intersect all balls with $\Omega$. 
For the last case, 
$B_{2R}^c \cap\Omega \times B_R^c \cap\Omega$, 
we have that $\eta(x)-\eta(z)$ vanishes and the integral is zero. For the case $B_{2R}\times B_{3R}^c$, we have $\abs{x-z} \geq R$ there. 
This gives 
\begin{align*}
&\int_{B_{2R}} \int_{B_{3R}^c\cap\Omega} \frac{|\eta(x)f(z) - \eta(z)f(z)|^2}{\abs{x-z}^{d+2-2s}}\,dz\,dx = \int_{B_{2R}} \int_{B_{3R}^c\cap\Omega} \frac{|\eta(x)f(z)|^2}{\abs{x-z}^{d+2-2s}}\,dz\,dx \\
&\qquad\quad\leq R^{-d-2+2s} \norm{\eta}_{L^\infty(B_R)}^2 \int_{B_{2R}} \int_{B_{3R}^c\cap\Omega} |f(z)|^2 dz dx \lesssim R^{-2+2s} \norm{\eta}_{L^\infty(B_R)}^2 \norm{f}_{L^2(\Omega)}^2.
\end{align*}
For the  integration over $B_{2R}^c\cap \Omega \times B_{R}$, we write using polar coordinates (centered at $z$)
\begin{align*}
&\int_{B_{2R}^c\cap \Omega} \int_{B_{R}} \frac{|\eta(z)f(z)|^2}{\abs{x-z}^{d+2-2s}}\,dz\,dx 
  = 
  \int_{B_{R}}|\eta(z)f(z)|^2 \int_{B_{2R}^c\cap\Omega} \frac{1}{\abs{x-z}^{d+2-2s}}\,dx\,dz \\
&\qquad\quad 
  \lesssim \int_{B_{R}}|\eta(z)f(z)|^2 \int_{R}^\infty \frac{1}{r^{3-2s}}\,dx\,dz 
  \lesssim R^{2s-2}\norm{\eta}_{L^\infty(B_R)}^2 \norm{f}_{L^2(\Omega)}^2.
\end{align*}
Finally, for the integration over $B_{2R} \times B_{3R}$, we use that $\abs{\eta(x)-\eta(z)} 
\leq \norm{\nabla \eta}_{L^\infty(B_R)} \abs{x-z}$ and polar coordinates (centered at $z$) to estimate
\begin{align*}
&\int_{B_{2R}} \int_{B_{3R}} \frac{|\eta(x)f(z) - \eta(z)f(z)|^2}{\abs{x-z}^{d+2-2s}}\,dz\,dx  
  \leq  \norm{\nabla \eta}_{L^\infty(B_R)}^2\int_{B_{3R}} |f(z)|^2 \int_{B_{2R}} \frac{1}{\abs{x-z}^{d-2s}}\,dx\,dz \\
&\qquad\quad \lesssim  \norm{\nabla \eta}_{L^\infty(B_R)}^2\int_{B_{3R}} |f(z)|^2 \int_{0}^{5R} r^{-1+2s}\,dr\,dz \lesssim  \norm{\nabla \eta}_{L^\infty(B_R)}^2  \norm{f}_{L^2(B_{3R})}^2 R^{2s}.
\end{align*}
The straightforward bound $\|\eta f\|_{L^2(\Omega)}\leq \|\eta\|_{L^\infty(B_R)} \| f\|_{L^2(\Omega)}$
concludes the proof.
\end{proof}

\section{Proof of Lemma~\ref{lemma:properties-of-H1alpha}}
\label{app:lemma:properties-of-H1alpha}
\begin{numberedproof}{of Lemma~\ref{lemma:properties-of-H1alpha}} 
The proof follows from the arguments given in \cite[Sec.~{3}]{KarMel19}; a more general development of Beppo-Levi spaces 
is given in \cite{deny-lions55}. 

\emph{Proof of (\ref{item:lemma:properties-of-H1alpha-i}):} 
Fix a (nondegenerate) hypercube $K  = \prod_{i=1}^{d+1} (a_i,b_i)$ with $a_{d+1} = 0$. Elements of the Beppo-Levi space 
$\operatorname{BL}^1_{\alpha}$ 
are locally in $L^2$, and one can equip the space ${\operatorname{BL}}^1_\alpha$ with the norm 
$\|U\|^2_{{\operatorname{BL}}^1_\alpha}:= \|U\|^2_{L^2_\alpha(K)} + \|\nabla U\|^2_{L^2_\alpha(\R^d \times \R_+)}$. 
Endowed with this norm, ${\operatorname{BL}}^1_\alpha$ is a Hilbert space and $C^\infty(\R^d \times [0,\infty)) \cap {\operatorname{BL}}^1_\alpha$ 
is dense, \cite[Lem.~{3.2}]{KarMel19}. On the subspace ${\operatorname{BL}}^1_{\alpha,0,\Omega}$ we show the norm equivalence 
$\|U \|_{{\operatorname{BL}}^1_\alpha} \sim \|\nabla U\|_{L^2_\alpha(\R^d\times \R_+)}$ using the bounded linear lifting operator 
${\mathcal E}: H^s(\R^d) \rightarrow H^1_\alpha(\R^d \times \Rpos)$ of \cite[Lem.~{3.9}]{KarMel19} 
and the norm equivalence of \cite[Cor.~{3.4}]{KarMel19}
\begin{align*}
\|\nabla U\|_{L^2_\alpha(\R^d \times \R_+)} & \leq 
\|U\|_{{\operatorname{BL}}^1_\alpha} \leq 
\|U - {\mathcal E} \operatorname{tr} U\|_{{\operatorname{BL}}^1_\alpha} + 
\|{\mathcal E} \operatorname{tr} U\|_{{\operatorname{BL}}^1_\alpha}  \\
 \overset{\text{\cite[Cor.~{3.4}]{KarMel19}}}&{\lesssim} \|\nabla( U - {\mathcal E}\operatorname{tr} U)\|_{L^2_\alpha(\R^d \times \R_+)} + 
\|{\mathcal E} \operatorname{tr} U\|_{{\operatorname{BL}}^1_\alpha}  \\
 \overset{\text{\cite[Lem.~{3.9}]{KarMel19}}}&{\lesssim} \|\nabla U \|_{L^2_\alpha(\R^d \times \R_+)} + \|\operatorname{tr} U\|_{H^s(\R^d)}  \\
 \overset{\operatorname{tr} U \in \widetilde{H}^s(\Omega), (\ref{eq:Htildet-vs-HtRd})}&{\lesssim} \|\nabla U \|_{L^2_\alpha(\R^d \times \R_+)} + |\operatorname{tr} U|_{H^s(\R^d)}  
 \stackrel{\text{\cite[Lem.~{3.8}]{KarMel19}}} {\lesssim} \|\nabla U \|_{L^2_\alpha(\R^d \times \R_+)}.  
\end{align*}
\emph{Proof of (\ref{item:lemma:properties-of-H1alpha-ii}):} 
From the fundamental theorem of calculus, 
we have for smooth univariate functions $v$ and $x \in (0,H)$ the estimate
$|v(x)| = |v(0) + \int_{t=0}^x v^\prime(t)\,dt| \lesssim |v(0)| + \sqrt{ \int_{t=0}^x t^\alpha |v^\prime(t)|^2\,dt }$.

Fix a closed hypercube $K' \subset \R^d$ of side length $d_{K'}>0$ with $K' \supset \Omega$.
 Define the translates $K_j:= d_{K'} j + K'$ for 
$j \in \Z^d$. 
For smooth $U$, we infer from the 1D estimate that 
\begin{align}
\label{eq:lemma:properties-of-H1alpha-10}
\|U\|_{L^2_\alpha(K' \times (0,H))} \leq C_{K'} \left( \|\nabla U\|_{L^2_\alpha(K' \times (0,H))} + \|\operatorname{tr} U\|_{L^2(K')} \right). 
\end{align}
By the density of 
$C^\infty(\R^d \times [0,\infty)) \cap {\operatorname{BL}}^1_\alpha$ in ${\operatorname{BL}}^1_\alpha$ 
from the proof of part (\ref{item:lemma:properties-of-H1alpha-i}),
the estimate 
(\ref{eq:lemma:properties-of-H1alpha-10}) holds for all $U \in {\operatorname{BL}}^1_\alpha$. 
By translation invariance of the norms and spaces, 
(\ref{eq:lemma:properties-of-H1alpha-10}) also holds 
for all $U \in {\operatorname{BL}}^1_\alpha $ and for all translates $K_j$, $j \in \Z^d$, with the same constant $C_{K'}$. 
For $U \in {\operatorname{BL}}^1_{\alpha,0,\Omega}$, we observe 
$\|\operatorname{tr} U\|_{L^2(K_0)} \leq \|\operatorname{tr} U\|_{\widetilde{H}^s(K_0)} 
\leq C_\Omega |\operatorname{tr} U|_{H^s(\R^d)}$ (cf.\ (\ref{eq:Htildet-vs-HtRd})) 
and $\operatorname{tr} U|_{K_j} = 0$ for $j \ne 0$. 
Hence, using the Kronecker $\delta_{j,0}$, we arrive at  
\begin{align*}
\label{eq:lemma:properties-of-H1alpha-20}
\|U\|_{L^2_\alpha(K_j \times (0,H))} \leq C_{K'} \left( \|\nabla U\|_{L^2_\alpha(K_j \times (0,H))} + C_\Omega \delta_{j,0} |\operatorname{tr} U|_{H^s(\R^d)} \right). 
\end{align*}
Since  $\R^d  = \cup_{j \in \Z^d} K_j$ and the intersection $K_j \cap K_{j'}$ is a set of measure zero for $j \ne j'$,  summation over all $j$ implies 
\begin{align*}
\|U\|_{L^2_\alpha(\R^d \times (0,H))} \lesssim \|\nabla U\|_{L^2_\alpha(\R^d \times (0,H))} + |\operatorname{tr} U|_{H^s(\R^d)}. 
\end{align*}
The proof is completed by noting $|\operatorname{tr} U|_{H^s(\R^d)} \lesssim \|\nabla U\|_{L^2_\alpha(\R^d \times \R_+)}$ by 
\cite[Lem.~{3.8}]{KarMel19}. 
\end{numberedproof}

\bibliography{bibliography}
\bibliographystyle{alpha}
\end{document}